\documentclass[12pt]{amsart}
\author{Elisheva Adina Gamse}
\thanks{The first author was partially supported by NSERC PostDoctoral Fellowship 488168.}
\author{Jonathan Weitsman}
\thanks{Both authors were partially supported by NSF grant DMS 12-11819}
\title{Relations in the cohomology ring of the moduli space of flat $SO(2n+1)$-connections on a Riemann surface}
\date{1 June 2018}
\usepackage{fullpage}
\usepackage{amsmath,amsfonts,amsthm,amssymb}
\usepackage{tikz}
\usetikzlibrary{patterns}
\usepackage{amsfonts}
\usepackage{mathrsfs}
\usepackage{graphicx}
\usepackage{amscd,amsbsy}
\usepackage[all]{xy}
\usepackage[colorlinks,plainpages,urlcolor=black,citecolor=black,linkcolor=black]{hyperref}
\usepackage{verbatim}
\usepackage{fancyhdr}
\usepackage{subcaption}
\newtheorem{thm}{Theorem}[section]
\newtheorem{prop}[thm]{Proposition}
\newtheorem{lem}[thm]{Lemma}
\newtheorem{cor}[thm]{Corollary}

\newtheorem{rmk}[thm]{Remark}
\theoremstyle{definition}
\newtheorem{defn}[thm]{Definition}
\newtheorem*{notn}{Notation}
\theoremstyle{remark}
\newtheorem*{ill}{Illustration}
\DeclareMathOperator{\Hom}{Hom}

\DeclareMathOperator{\Stab}{Stab}

\renewcommand{\ggg}{\mathfrak{g}}
\newcommand{\qqq}{\mathbb{Q}}
\newcommand{\ccc}{\mathbb{C}}
\newcommand{\zzz}{\mathbb{Z}}
\newcommand{\rrr}{\mathbb{R}}
\newcommand{\nnn}{\mathbb{N}}
\newcommand{\cphi}{\mathbb{C}_{(\phi)}}
\newcommand{\spmijx}{s^{\pm}_{ij}(x)}
\newcommand{\spijx}{s^{+}_{ij}(x)}
\newcommand{\smijx}{s^{-}_{ij}(x)}
\newcommand{\lpmij}{L^\pm_{ij}}
\newcommand{\lpij}{L^+_{ij}}
\newcommand{\lmij}{L^-_{ij}}
\newcommand{\ypmij}{y^\pm_{ij}}
\newcommand{\ypij}{y^+_{ij}}
\newcommand{\ymij}{y^-_{ij}}
\newcommand{\qyx}{\mathbb{Q}[Y(X)]}
\newcommand{\qymx}{\mathbb{Q}[Y^-(X)]}
\newcommand{\qyzx}{\mathbb{Q}[Y_z(X)]}
\newcommand{\qyxmz}{\mathbb{Q}[Y(X \setminus \{z\})}
\newcommand{\setgens}{\{a_1, \ldots, a_g, b_1, \ldots, b_g\}}
\newcommand{\genss}{a_1, \ldots, a_g, b_1, \ldots, b_g}
\newcommand{\sehfwz}{\{e_1, \ldots, e_h, f_1, \ldots, f_w, z\}}
\newcommand{\qyehz}{\mathbb{Q}[Y(\{e_1, \ldots, e_h, z\})]}
\newcommand{\qyeh}{\mathbb{Q}[Y(\{e_1, \ldots, e_h\})]}
\newcommand{\qyfwz}{\mathbb{Q}[Y(\{f_1, \ldots, f_w, z\})]}
\newcommand{\qyfw}{\mathbb{Q}[Y(\{f_1, \ldots, f_w\})]}
\newcommand{\qymbx}{\mathbb{Q}[Y^-_B(X)]}
\newcommand{\bbb}{\mathcal{B}}
\newcommand{\seh}{\{e_1, \ldots, e_h\}}
\newcommand{\sehz}{\{e_1, \ldots, e_h,z\}}
\newcommand{\sfw}{\{f_1, \ldots, f_w\}}
\newcommand{\sfwz}{\{f_1, \ldots, f_w,z\}}
\newcommand{\eeh}{e_1, \ldots, e_h}
\newcommand{\ffw}{f_1, \ldots, f_w}
\newcommand{\bx}{\mathcal{B}[X]}
\newcommand{\bxmz}{\mathcal{B}[X\setminus \{z\}]}
\newcommand{\xmz}{X \setminus \{z\}}
\newcommand{\ymbx}{Y^-_B(X)}
\newcommand{\behz}{\mathcal{B}[\{e_1, \ldots, e_h, z\}]}
\newcommand{\bfwz}{\mathcal{B}[\{f_1, \ldots, f_w, z\}]}
\newcommand{\squareminus}{\begin{tikzpicture} \node[draw, rectangle, blue] at (0,0){$-$}; \end{tikzpicture}}
\newcommand{\squareplus}{\begin{tikzpicture} \node[draw, rectangle, red] at (0,0){$+$}; \end{tikzpicture}}
\newcommand{\yecij}{y^{\epsilon_C(i,j)}_{ij}}
\newcommand{\ymecij}{y^{-\epsilon_C(i,j)}_{ij}}
\begin{document}

\begin{abstract} We consider the moduli space of flat $SO(2n+1)$-connections (up to gauge transformations) on a Riemann surface, with fixed holonomy around a marked point. There are natural line bundles over this moduli space; we construct geometric representatives for the Chern classes of these line bundles, and prove that the ring generated by these Chern classes vanishes below the dimension of the moduli space, generalising a conjecture of Newstead. \end{abstract}
\maketitle
\section{Introduction}

Let $G$ be a compact Lie group with Lie algebra $\ggg$, and pick a maximal torus $T\subset G$. Spaces of the form $(\Hom \pi_1(\Sigma),G)/G$, where $\Sigma$ is a Riemann surface of genus $ g \geq 2$, arise in various branches of geometry. Such spaces have interpretations as the moduli space of flat connections up to gauge transformations (see \cite{ljjwtoric,ljjw, yoshida}); they also occur as a building block in the topological quantum field theoretical construction of invariants of 3-manifolds with boundary $\Sigma$. See also the related discussions of character varieties (\cite{sikora,simpson}), of parabolic Higgs bundles (\cite{hyperpolygon}) and of polygon spaces (\cite{polygon}).

A related space is obtained by marking a point $p \in \Sigma$ and prescribing the holonomy of the connections around that point. That is, fix a generator $c \in \pi_1(\Sigma \setminus p)$ that represents a small curve around $p$, and for $t \in T$, let $S_{g}(t):= \{\rho \in \Hom(\pi_1(\Sigma \setminus p), G \vert \rho(c) \sim t \} / G$, where $\sim$ denotes conjugacy in $G$. In the case where the holonomy $\zeta$ lies in the centre of $G$, the spaces $S_g(\zeta)$ are moduli spaces of stable holomorphic vector bundles over $\Sigma$ (see for example \cite{ab,tw,ek,ek2,kirwan,kirwan2}).

We will consider the space $S_g(t)$, where $t$ is a generic torus element with $\Stab t = T$. In this paper, we will take $G=SO(2n+1)$. Consider the torus bundle $V_g(t) \to S_g(t)$ given by $V_g(t):=\{\rho \in \Hom(\pi_1(\Sigma \setminus p),G) \vert \rho(c) =t\}$. If $\phi \in \Phi(\ggg)$ is a root of $\ggg$, let $L_\phi$ be the line bundle associated to $V_g(t)$ via the torus representation with weight $\phi$. We will construct geometric representatives for the first Chern classes of these line bundles. By considering these geometric representatives, we are able to identify several particular products of Chern classes which vanish in $H^*(S_g(t))$; we also give a combinatorial proof that any monomial in the $c_1(L_\phi)$ is equivalent in $H^*(S_g(t))$ to a combination of those particular monomials and hence also vanishes. Our geometric representatives are analogous to Schubert cycles for flag manifolds; however, a key difference is that there is no canonical complex structure on $S_g(t)$, and our geometric representatives for Chern classes will not generally be complex subvarieties of $S_g(t)$ with respect to an arbitrary choice of K\"ahler structure. Thus a feature of our topological approach is that it enables us to  make use of these particular geometric representatives which would not show up under an algebraic geometric treatment of the subject. 

This paper builds on our earlier work \cite{aj}, where we used a similar approach for the case $G=SU(n)$, making $S_g(t)$ the moduli space of parabolic holomorphic vector bundles over $\Sigma$. (For a different approach to finding generators and relations for the cohomology of this moduli space, see \cite{ek2}.) This itself was based on the earlier paper \cite{jw}, which used this geometric approach in the case $G=SU(2)$ to provide a geometric proof to a conjecture of Newstead (\cite{newstead}). 

Let us now take $G=SO(2n+1)$, and fix generators $a_1,\ldots, a_g, b_1, \ldots, b_g, c$ for the fundamental group $\pi_1(\Sigma \setminus \{p\})$, such that $c$ represents the boundary of $\Sigma \setminus \{p\}$ and $\prod_{l=1}^g[a_l, b_l]=c$. Recall that the set of roots of $\ggg$ is $\Phi(\ggg)=\{\pm(\eta_i + \eta_j) \mid 1 \leq i \leq j \leq n\}\cup \{\pm(\eta_i - \eta_j) \mid 1 \leq i < j \leq n\}$. Choose the maximal torus $T \subset G$ consisting of elements 
\[
\left(\begin{array}{cccccc}
\cos \theta_1 & -\sin \theta_1 & &&&\\
\sin \theta_1 & \cos \theta_1 &&&& \\
&& \ddots &&& \\
&&& \cos \theta_n & -\sin \theta_n & \\
&&& \sin \theta_n & \cos \theta_n & \\
&&&&&1\\
\end{array}\right);
\]
for ease of notation such elements will be denoted $(\theta_1, \ldots, \theta_n)$. 
\begin{defn}\label{generic}
We say an element $t = (\theta_1, \ldots, \theta_n)$ is \emph{generic} if $\Stab t = T$, and if the only relation $\lambda_1 \theta_1 + \cdots + \lambda_n \theta_n \in 2\pi \zzz $ with $\lambda_i \in \{\pm 1, 0\}$ is the trivial relation $\lambda_1 = \cdots = \lambda_n = 0$. 
\end{defn}
Let $t \in T$ be generic, and set
\[S_{n,g}(t) := \{ \rho \in \Hom(\pi_1(\Sigma \setminus p), SO(2n+1))\mid \rho(c) \sim t\}/G.
\]
Let $\ccc_{(\pm(\eta_i \pm \eta_j))}$ denote the 1-dimensional torus representation 
\begin{align*}
T \times \ccc_{(\pm(\eta_i \pm \eta_j))} & \to \ccc_{(\pm(\eta_i \pm \eta_j))} \\
(\theta_1, \ldots, \theta_n) \cdot z &\mapsto e^{\pm \sqrt{-1}(\eta_i \pm \eta_j)}z.
\end{align*}
For $\phi \in \Phi(\ggg)$ a root of $\ggg$, we consider line bundles $L_\phi:=(V_g(t)\times \ccc_{(\phi)})/T \to V_g(t)/T=S_{n,g}(t)$, where the quotient is by the diagonal $T$-action. We will also denote $L_{(\eta_i \pm \eta_j)}$ by $L_{ij}^\pm$; observe that $L_{-(\eta_i \pm \eta_j)}\cong (L_{\eta_i \pm \eta_j})^*$.

\begin{thm} \label{main} The product $\prod_{\phi \in \Phi(\ggg)}c_1(L_\phi)^{k_\phi} \in H^*(S_{n,g}(t);\qqq)$ vanishes whenever $\sum_{\phi \in \Phi(\ggg)}k_\phi \geq 2gn^2+ \frac{1}{2}(n-1)(n-2) $.\end{thm}

The dimension of $S_{n,g}(t)$ is $2gn(2n+1)-2n(n+1)$, so when $g \geq \frac{3n}{2}$, our theorem shows that the top cohomology classes of $S_{n,g}(t)$ must be generated by elements other than the $c_1(L_\phi)$. 

We begin in the case $n=1$, the proof for which is analogous to Weitsman's proof (\cite{jw}) of the Newstead conjecture.  When $n=1$, we have just the one pair of roots $\pm\phi$ of $\mathfrak{so}(3)$, and the corresponding pair of line bundles $L_\phi$ and $L_{-\phi}=L_\phi^*$; of course $c_1(L_\phi)=-c_1(L_\phi^*)=-c_1(L_{-\phi})$. In this case, the theorem is that $c_1(L_\phi)^{2g}=0$, as follows:

\begin{prop} Let $G=SO(3)$. Let $t \in T$ be generic, and let $L \to S_{1,g}(t)$ be the line bundle associated to $V_g(t)$ by the representation $\left( \begin{array}{ccc} \cos \theta & -\sin \theta & \\ \sin \theta & \cos \theta & \\ &&1 \end{array} \right)\cdot z = e^{2i\theta}z$. Then $(c_1(L))^{2g}=0$ in $H^{4g}(S_{1,g}(t); \qqq)$.
\end{prop}

\begin{proof}
For $x  \in \{a_1, \ldots, a_g, b_1, \ldots, b_g\}$, consider the section $s_x$ of $L$ induced by the $T$-equivariant map 
\begin{align*} f_x: V_g(t) &\to \ccc \\ \rho &\mapsto (\rho(x))_{11}-(\rho(x))_{22} + \sqrt{-1}((\rho(x))_{21}+(\rho(x))_{12})
\end{align*} (where the subscript $ij$ denotes the $(i,j)^{th}$ matrix entry). If $s_x=0$, then 
\[
\rho(x) = \left( \begin{array}{ccc} a & b & c \\ -b & a & d \\ e & f&g \end{array} \right) \in SO(3).
\]
 But since $\rho(x) \in SO(3)$, we also have $\vert c \vert = \vert d \vert$ and $cd=0$, thus $c=d=0$. Similarly $e=f=0$, and so $a^2 + b^2 = 1$ and $g=1$; that is, $\rho(x) \in T$. Suppose $\rho(x)=0$ for all $x \in \{a_1, \ldots, a_g, b_1, \ldots, b_g\}$. Then $\rho(a_1), \ldots, \rho(a_g), \rho(b_1), \ldots, \rho(b_g) \in T$. But then $\rho(c) = \prod_{i=1}^g [\rho(a_i), \rho(b_i)]=1\neq t$, so in fact there are no such $\rho$ in $V_g(t)$. Hence the section $(s_{a_1},\ldots, s_{a_g}, s_{b_1}, \ldots, s_{b_g})$ of $L^{\oplus 2g}$ is nowhere zero, so $c_{2g}(L^{\oplus 2g})=(c_1(L))^{2g}=0$.
\end{proof}

When $n > 1$, the combinatorics of the vanishing loci of the relevant sections becomes more complicated. The key idea in the proof of Theorem \ref{main} is similar to that in our earlier paper \cite{aj}, but the combinatorics required in this case is more intricate.

The outline of the rest of this paper is as follows. 

In section \ref{five}, we prove the theorem for the case $G=SO(5)$. As in \cite{aj}, the proof begins by identifying specific products of Chern classes which must vanish because there are sections (constructed explicitly) of the relevant line bundles with no common zeros. The proof is completed by showing that any monomial of sufficiently high degree is equivalent to a combination of these products of Chern classes which have already been shown to vanish. In section \ref{purple} we allow $n$ to be an arbitrary positive integer, exhibit some collections of sections of the appropriate line bundles with no common zeros, and conclude that the corresponding products of Chern classes vanish. In section \ref{heart} we turn to the combinatorial heart of the argument, showing that every product of Chern classes of sufficient degree is equivalent in $H^*(S_{n,g}(t))$  to a combination of those shown to vanish in section \ref{purple}. 

\subsection*{Acknowledgements} We would like to thank the anonymous referee for the extremely thorough and detailed comments. 

\subsection*{Authors' Contributions} E.A.G. and J.W. proposed the project and carried out research.  E.A.G. wrote the draft of the manuscript.  E.A.G. and J.W. approved the final version.

\section{$G=SO(5)$} \label{five}

We have chosen to devote a section to proving our result in the case $G=SO(5)$. This is intended to provide intuition for the general case in a more tractable context.
Let $G=SO(5)$ and let $T\subset G$ be the maximal torus 
\[
T = \left\{\left(\begin{array}{cccccc}
\cos \theta_1 & -\sin \theta_1 & &&&\\
\sin \theta_1 & \cos \theta_1 &&&& \\
&&& \cos \theta_2 & -\sin \theta_2 & \\
&&& \sin \theta_2 & \cos \theta_2 & \\
&&&&&1\\
\end{array}\right)\right\} \subset G;
\]
to simplify notation we will also denote elements of $T$ by $(\theta_1,\theta_2)$. Choose an element $t \in T$ that is generic in the sense that $\theta_1 \pm \theta_2 \notin 2\pi\zzz$. Recall that the roots of $\ggg$ are $\pm(2\eta_1, 2\eta_2, \eta_1 + \eta_2, \eta_1 - \eta_2)$. If $\phi = \pm(\eta_j \pm \eta_k)$ we denote by $\ccc_{(\phi)}$ the one dimensional torus representation $(\theta_1,\theta_2) \cdot z = e^{\pm i (\theta_j \pm \theta_k)}z$. Consider the line bundles $L_\phi = V_g(t) \times_T \cphi \to V_g(t)/T = S_{2,g}(t) $ for $\phi \in \Phi(G)$. Observe that $L_{-\phi} \cong L^*_{\phi}$, and $L_{\phi + \psi} \cong L_\phi \otimes L_\psi$. Hence $c_1(L^*_\phi)=-c_1(L_\phi)$, and $c_1(L_{\phi + \psi}) = c_1(L_\phi) + c_1(L_\psi)$. In this case, Theorem \ref{main} says 
\begin{prop} \label{propfive} The product $\prod_{\phi \in \Phi(G)} c_1 (L_\phi)^{k_\phi}$ vanishes whenever $\sum_{\phi \in \Phi(G)}k_\phi \geq 8g$. 
\end{prop}
As in the case of $G=SO(3)$, and following the strategy for $G=SU(n)$ in \cite{aj}, we prove this by identifying collections of sections of the $L_\phi$ with no common zeros, thus observing that the corresponding products of first Chern classes $c_1(L_\phi)$ vanish. We will identify three such vanishing products, and then prove that any other monomial of degree at least $8g$ is equivalent in $H^*(S_{2,g}(t))$ to a combination of those three.

\begin{defn} For $x \in \{a_l, b_l \mid 1 \leq l \leq g \} $, let $s^\pm_{ij}(x)$ be the section of $L_{\eta_i \pm \eta_j}$ induced by the $T$-equivariant map 
\begin{align*}
V_g(t) \to & \ccc_{(\eta_i \pm \eta_j)}\\
\rho \mapsto & (\rho(x))_{2i-1,2j-1} \mp (\rho(x))_{2i,2j} + \sqrt{-1}((\rho(x))_{2i,2j-1}\pm(\rho(x))_{2i-1,2j}).
\end{align*}
(Note that $s^-_{ii}(x)$ is a section of the trivial bundle $L_0$.)
\end{defn}
Observe that if $\spmijx$ vanishes, then the $(i,j)^{th}$ 2-by-2 block in $\rho(x)$ takes the form $\left(\begin{array}{cc} a & b \\ \mp b & \pm a \end{array}\right)$; if for any given $i,j$, both $\spijx$ and $\smijx$ vanish, then the $(i,j)^{th}$ 2-by-2 block in $\rho(x)$ is zero. 

To simplify notation, will will write $c^\pm_{ij} := c_1(L^\pm_{ij})$

\begin{lem} The monomial $z_1 := (c^+_{11}c^+_{12}c^-_{12})^{2g}$ vanishes in $H^{12g}(S_{2,g}(t))$.
\end{lem}
\begin{proof}Let $x \in \{a_l, b_l\}$, and consider the sections $s^+_{11}(x), s^+_{12}(x), s^-_{12}(x)$. If these all vanish, then 
$
\rho(x) = \left(\begin{array}{ccccc}
a & b &0 &0&c\\
-b & a &0&0&d \\
*&*&*& *& *\\
*&*&*& * & * \\
*&*&*&*&*\\
\end{array}\right)$, for some $a,b,c,d \in \rrr$. Since $\rho(x) \in SO(5)$, $c=d=0$, so in fact 
$\rho(x) = \left(\begin{array}{ccccc}
a & b &0 &0&0\\
-b & a &0&0&0 \\
0&0&*& *& *\\
0&0&*& * & * \\
0&0&*&*&*\\
\end{array}\right)$, for some $a, b \in \rrr$ with $a^2 + b^2 =1$. If all $6g$ sections $s^+_{11}(x), s^+_{12}(x), s^-_{12}(x)$, for $x = a_1, \ldots, a_g, b_1, \ldots, b_g$ vanish, then 
\begin{align*}
\rho(c) &= \prod_{l=1}^g[\rho(a_l),\rho(b_l)]\\ &= \prod_{l=1}^g \left[\left(\begin{array}{ccccc}
\cos \alpha_l & -\sin \alpha_l &0 &0&0\\
\sin \alpha_l & \cos \alpha_l &0&0&0 \\
0&0&*& *& *\\
0&0&*& * & * \\
0&0&*&*&*\\
\end{array}\right),\left(\begin{array}{ccccc}
\cos \beta_l & -\sin \beta_l &0 &0&0\\
\sin \beta_l & \cos \beta_l &0&0&0 \\
0&0&*& *& *\\
0&0&*& * & * \\
0&0&*&*&*\\
\end{array}\right) \right] \\ &= \left(\begin{array}{ccccc}
1 & 0 &0 &0&0\\
0 & 1 &0&0&0 \\
0&0&*& *& *\\
0&0&*& * & * \\
0&0&*&*&*\\
\end{array}\right) \neq t 
\end{align*}
 (since $t$ was chosen to be generic). But $\rho(c) = t$ for $\rho \in V_g(t)$, and so these $6g$ sections have no common zeros; thus $(c^+_{11} c^+_{12}c^-_{12})^{2g}=0$ as claimed. 
\end{proof}

Similarly, $z_2:=(c^+_{22} c^+_{12}c^-_{12})^{2g}=0$. 

\begin{lem}\label{fiveplus}Suppose $A_l \in SO(5)$ for $1 \leq l \leq g$, and each $A_l$ has the form $\left(\begin{array}{ccccc}
&&&&0\\
\multicolumn{2}{c}{\smash{\raisebox{.5\normalbaselineskip}{$R_{\alpha_l}$}}}&\multicolumn{2}{c}{\smash{\raisebox{.5\normalbaselineskip}{$R_{\beta_l}$}}}&0\\
&&&&0\\
\multicolumn{2}{c}{\smash{\raisebox{.5\normalbaselineskip}{$R_{\gamma_l}$}}}&\multicolumn{2}{c}{\smash{\raisebox{.5\normalbaselineskip}{$R_{\delta_l}$}}}&0\\
0&0&0&0&1
\end{array}\right)$, where $R_\theta$ represents the 2-by-2 block $\left(\begin{array}{cc}\cos \theta & - \sin \theta \\ \sin \theta & \cos \theta \\\end{array}\right)$. Then $\prod_{l=1}^g [A_{2l-1},A_{2l}]$ cannot be a generic torus element. 
\end{lem}

\begin{proof} Consider the map 
\begin{align*}
\kappa: U(2) &\hookrightarrow SO(5) \\
\kappa: \left(\begin{array}{cc}a & b\\ c & d \\\end{array}\right) & \mapsto \left(\begin{array}{ccccc} \Re a & -\Im a & \Re b & - \Im b & 0 \\\Im a & \Re a & \Im b & \Re b & 0 \\ \Re c & -\Im c & \Re d & - \Im d & 0 \\\Im c & \Re c & \Im d & \Re d & 0 \\ 0&0&0&0&1\\ \end{array}\right).
\end{align*} This is an injective homomorphism, and $\kappa(M) \in T \iff M$ is diagonal in $U(2)$. Each $A_l$ is in the image of $\kappa$; say $A_l=\kappa(M_l)$, for $M_l \in U(2)$. Then 
\begin{align*}
\prod_{l=1}^g [ A_{2l-1}, A_{2l}] &= \prod_{l=1}^g[\kappa(M_{2l-1}), \kappa(M_{2l})] \\
&= \kappa \left( \prod_{l=1}^g [M_{2l-1},M_{2l}]\right).
\end{align*}
Observe that $M=\prod_{l=1}^g [M_{2l-1}, M_{2l}]$ has determinant 1. If $M$ is not diagonal, then $\kappa(M) \neq t$. If $M$ is diagonal, then it must be $\left(\begin{array}{cc}e^{i\theta} & \\ & e^{-i\theta} \end{array} \right)$ for some $\theta$, so $\kappa(M) = \left(\begin{array}{ccc}R_\theta && \\ & R_{-\theta} & \\ &&1 \end{array} \right)$ is not generic. Thus in particular, $\prod_{l=1}^g [A_{2l-1}, A_{2l}] \neq t$. 
\end{proof}

\begin{cor}
The monomial $p=(c^+_{11}c^+_{12}c^+_{21}c^+_{22})^{2g}$ vanishes in $H^{16g}(S_{2,g}(t))$. 
\end{cor}

\begin{proof} For each $x \in \{a_l, b_l \mid 1 \leq l \leq g\}$, consider the sections $s^+_{11}(x), s^+_{12}(x), s^+_{21}(x)$, and $s^+_{22}(x)$. If these were all to vanish, then each $\rho(x)$ would have the form $\left( \begin{array}{ccccc} a&b&c&d&\\ -b&a&-d&c&\\e&f&g&h&\\-f&e&-h&g&\\&&&&1\end{array}\right)$, where $a,b,c,d,e,f \in \rrr$. By Lemma \ref{fiveplus}, if these $8g$ sections all vanish, then $\prod_{l=1}^g[\rho(a_l),\rho(b_l)]\neq t$. Thus these $8g$ sections have no common zeros, so $(c^+_{11}c^+_{12}c^+_{21}c^+_{22})^{2g}=0$.  
\end{proof}

\begin{lem}
\label{fivepm}
Suppose $B_k \in SO(5)$ for $1 \leq k \leq 2g$, and each $B_k$ has the form $\left( \begin{array}{ccccc} a&b&c&d&\\ -b&a&d&-c&\\e&f&g&h&\\f&-e&-h&g&\\&&&&1\end{array}\right)$. Then $\prod_{k=1}^g[B_{2k-1},B_{2k}]\neq t$.\end{lem}

\begin{proof} Let $E = \left(\begin{array}{ccccc}0&1&&&\\ 1&0&&&\\ &&1&&\\&&&1&\\&&&&1\end{array}\right)$. Observe that $EB_kE^{-1}$ has the form described in Lemma \ref{fiveplus}. So by Lemma \ref{fiveplus}, $\prod_{k=1}^g[EB_{2k-1}E^{-1},EB_{2k}D^{-1}]$ cannot be a generic torus element. If $\prod_{k=1}^g[B_{2k-1},B_{2k}]= t$, then $\prod_{k=1}^g[EB_{2k-1}E^{-1},EB_{2k}D^{-1}]=EtE^{-1}$. Note that 
\[E \left( \begin{array}{ccc}R_{\theta_1}&&\\&R_{\theta_2}&\\&&1\end{array}\right)E^{-1}= \left( \begin{array}{ccc}R_{-\theta_1}&&\\&R_{\theta_2}&\\&&1\end{array}\right), \]
so $EtE^{-1}$ is a generic torus element. This is impossible by Lemma \ref{fiveplus}, hence $\prod_{k=1}^g[B_{2k-1},B_{2k}]\neq t$. 
\end{proof}

\begin{cor}
The monomial $q = (c^+_{11}c^+_{22}c^-_{12}c^-_{21})^{2g}$ vanishes in $H^{16g}(S_{2,g}(t))$. 
\end{cor}

\begin{proof} Consider the sections $s^+_{11}(x), s^+_{22}(x), s^-_{12}(x)$, and $s^-_{21}(x)$ for each $x \in \{ a_l, b_l \vert 1 \leq l \leq g\}$. If these all vanish, then each $\rho(x)$ has the form described in Lemma \ref{fivepm}, and so $\prod_{l=1}^g[\rho(a_l),\rho(b_l)]\neq t = \rho(c)$. So these $8g$ sections have no common zeros, so $(c^+_{11}c^+_{22}c^-_{12}c^-_{21})^{2g}=0$.
\end{proof}

\begin{lem} 
Let $0 \leq m \leq 2g$. The monomials 
\begin{equation} \label{typeone}y^+_m:=(c^+_{11})^{2g}(c^+_{22})^{2g}(c^+_{12})^{2g+m}(c^-_{12})^{2g-m}
\end{equation} and 
\begin{equation} \label{typetwo} y^-_m:=(c^+_{11})^{2g}(c^+_{22})^{2g}(c^+_{12})^{2g-m}(c^-_{12})^{2g+m}
\end{equation}vanish in $H^{16g}(S_{2,g}(t))$.
\end{lem}

\begin{proof}For fixed $x$, we have seen above that the sections $s^+_{11}(x), s^+_{12}(x), s^+_{21}(x)$, and $s^+_{22}(x)$ all vanish when $\rho(x)$ has the form 
\begin{equation}\label{matformp}\left( \begin{array}{ccccc} a&b&c&d&\\ -b&a&-d&c&\\e&f&g&h&\\-f&e&-h&g&\\&&&&1\end{array}\right).\end{equation} Observe that if the sections $s^+_{11}(x), s^-_{12}(x), s^+_{12}(x)$, and $s^+_{22}(x)$ all vanish, then $\rho(x)$ has the form $\left( \begin{array}{ccccc} a&b&0&0&0\\ -b&a&0&0&0\\0&0&c&d&0\\0&0&-d&c&0\\0&0&0&0&1\end{array}\right)$, which is in particular also in the form \eqref{matformp}. Thus if $X_m$ is any $m$-element subset of $\{a_1, \ldots, a_g, b_1, \ldots, b_g\}$, then the collection of sections 
\begin{align*}
s^+_{11}(x), s^+_{22}(x), s^+_{12}(x) & \text{ for all } x \in \{a_1, \ldots, a_g, b_1, \ldots, b_g\},\\
s^-_{12}(x) & \text{ for all } x \in X^c_m, \\
s^+_{21}(x) & \text{ for all } x \in X_m
\end{align*} has no common zeros, so the monomial \eqref{typeone} vanishes.

Similarly, the collection of sections $s^+_{11}(x), s^+_{22}(x), s^-_{12}(x)$, and $s^-_{21}(x)$ vanishes when $\rho(x)$ has the form $\left( \begin{array}{ccccc} a&b&c&d&\\ -b&a&d&-c&\\e&f&g&h&\\f&-e&-h&g&\\&&&&1\end{array}\right)$, and the collection $s^+_{11}(x), s^+_{22}(x), s^-_{12}(x)$, and $s^+_{12}(x)$ vanishes when $\rho(x)$  has the form $\left( \begin{array}{ccccc} a&b&0&0&0\\ -b&a&0&0&0\\0&0&c&d&0\\0&0&-d&c&0\\0&0&0&0&1\end{array}\right)$, so by Lemma \ref{fivepm} the collection of sections \begin{align*}
s^+_{11}(x), s^+_{22}(x), s^-_{12}(x) & \text{ for all } x \in \{a_1, \ldots, a_g, b_1, \ldots, b_g\},\\
s^+_{12}(x) & \text{ for all } x \in X^c_m, \\
s^-_{21}(x) & \text{ for all } x \in X_m
\end{align*} has no common zeros, so the monomial \eqref{typetwo} vanishes.
\end{proof}

\begin{prop}
Suppose $\zeta$ is a monomial in the $c^\pm_{ij}$ with degree at least $8g$. Then $\zeta$ is a combination of monomials $z_1, z_2, y^+_m$, and $y^-_m$. 
\end{prop}

\begin{proof}Use the relations 
\begin{align*}\label{fiverelns}
c^+_{12}&=\frac{1}{2}(c^+_{11}+c^+_{22}) \\
c^-_{12} &= \frac{1}{2}(c^+_{11}-c^+_{22})\\
c^-_{ij} &= -c^-_{ji}
\end{align*} to write $\zeta$ as a sum of terms $(c^+_{11})^a(c^+_{22})^b$, where $a+b \geq 8g$, and consider each term separately. We must have $a \geq 2g$ or $b \geq 2g$; without loss of generality assume $a \geq 2g$. Rewrite this term as $(c^+_{11})^{2g}(2c^+_{12}-c^+_{22})^{a-2g}(c^+_{22})^b$, and consider each term $\lambda(c^+_{11})^{2g}(c^+_{12})^{a-2g-m}(c^+_{22})^{b+m}$. Note that $a-2g+b \geq 6g$, so either $a-2g-m \geq 2g$ or $b+m \geq 2g$. \begin{itemize} \item If $a-2g-m \geq 2g$, rewrite this term as $\lambda'(c^+_{11})^{2g}(c^+_{12})^{2g}(c^-_{12}+c^+_{22})^{a-m-4g}(c^+_{22})^{b+m}$, and consider each term $\tilde{\lambda}(c^+_{11})^{2g}(c^+_{12})^{2g}(c^-_{12})^{a-m-4g-s}(c^+_{22})^{b+m+s}$ in the resulting expansion. Note that $a+b-4g \geq 4g$, so either $a-m-4g-s \geq 2g$ or $b+m+s \geq 2g$. \begin{itemize} \item If $a-m-4g-s \geq 2g$, this term is a multiple of $z_1$. \item If $b+m+s \geq 2g$, rewrite this term as $\bar{\lambda}(c^+_{11})^{2g}(c^+_{22})^{2g}(c^+_{12})^{2g}(c^-_{12})^{a-m-4g-s}(c^+_{12}-c^-_{12})^{b+m+s-2g}$. Each term in the expansion of this polynomial is a multiple of some $y^\pm_m$.  \end{itemize} \item If $b+m \geq 2g$, rewrite the term as $\lambda(c^+_{11})^{2g}(c^+_{22})^{2g}(c^+_{12})^{a-2g-m}(c^+_{12}-c^-_{12})^{b+m-2g}$. Each term in the expansion of this polynomial is a multiple of some $y^\pm_m$.\end{itemize}Hence $\zeta$ is equal to a sum of multiples of the monomials $z_1, z_2$ and $y^\pm_m$ as claimed.
\end{proof}

\begin{cor} 
$\prod_{\phi \in \Phi(SO(5))}c_1(L_\phi)^{k_\phi}$ vanishes whenever $\sum_{\phi \in \Phi(SO(5))}k_\phi \geq 8g.$ \qed 
\end{cor}
\section{The general case} \label{purple}
We now begin our study of the general case $G=SO(2n+1)$. As we did for $SO(5)$ in the previous section, we will show that some particular products of the $c_1(L_\phi)$, for roots $\phi$ of $\mathfrak{so(2n+1)}$, vanish, by finding sections of the $L_\phi$ with no common zeros. 

A $T$-equivariant map $V_g(t) \to \ccc_{(\phi)}$ induces a section of $L_\phi = (V_g(t) \times \cphi)/T \to V_g(t)/T=S_{n,g}(t)$. Let $x \in \{a_1, \ldots, a_g, b_1, \ldots, b_g\}$ be one of our chosen generators of $\pi_1(\Sigma\setminus \{p\})$ other than $c$, and consider the maps
\begin{align*}
f^\pm_{ij}(x): V_g(t)&\to \ccc_{(\eta_i \pm \eta_j)} \\
\rho & \mapsto (\rho(x))_{2i-1,2j-1}\mp (\rho(x))_{2i,2j}+ \sqrt{-1}((\rho(x))_{2i,2j-1} \pm (\rho(x))_{2i-1,2j}). 
\end{align*} These maps are $T$-equivariant and induce sections $\spmijx$ of $\lpmij$. The nature of the $T$ action on $SO(2n+1)$ is such that top left $2n$-by-$2n$ corner of elements of $SO(2n+1)$ can be divided into 2-by-2 matrices on which the behaviour of the torus action can be considered separately. Thus it will often be convenient to write elements of $SO(2n+1)$ in the form 
$\left(\begin{array}{ccccccc}
&&&&&&b_1\\
\multicolumn{2}{c}{\smash{\raisebox{.5\normalbaselineskip}{$A_{11}$}}}&\multicolumn{2}{c}{\smash{\raisebox{.5\normalbaselineskip}{$\cdots$}}}&\multicolumn{2}{c}{\smash{\raisebox{.5\normalbaselineskip}{$A_{1n}$}}}&b_2\\
\multicolumn{2}{c}{\vdots}&\multicolumn{2}{c}{\ddots}&\multicolumn{2}{c}{\vdots}&\vdots\\
&&&&&&b_{2n-1}\\
\multicolumn{2}{c}{\smash{\raisebox{.5\normalbaselineskip}{$A_{n1}$}}}&\multicolumn{2}{c}{\smash{\raisebox{.5\normalbaselineskip}{$\cdots$}}}&\multicolumn{2}{c}{\smash{\raisebox{.5\normalbaselineskip}{$A_{nn}$}}}&b_{2n}\\

c_1&c_2&\multicolumn{2}{c}{{\cdots}}&c_{2n-1}&c_{2n}&d
\end{array}\right)$, adopting the notational convention that capital letters denote 2-by-2 arrays whereas lowercase letters represent real numbers. 

The next lemma is the direct generalisation of Lemma \ref{fiveplus} to the general case.
\begin{lem}
\label{genplus} 
Suppose $M_l \in SO(2n+1)$ for $1 \leq l \leq 2g$, and each $M_l$ has the form 
\begin{equation}\label{plusmat}
\left(\begin{array}{ccccccc}
&&&&&&0\\
\multicolumn{2}{c}{\smash{\raisebox{.5\normalbaselineskip}{$R^{11}_l$}}}&\multicolumn{2}{c}{\smash{\raisebox{.5\normalbaselineskip}{$\cdots$}}}&\multicolumn{2}{c}{\smash{\raisebox{.5\normalbaselineskip}{$R^{1n}_l$}}}&0\\
\multicolumn{2}{c}{\vdots}&\multicolumn{2}{c}{\ddots}&\multicolumn{2}{c}{\vdots}&\vdots\\
&&&&&&0\\
\multicolumn{2}{c}{\smash{\raisebox{.5\normalbaselineskip}{$R^{n1}_l$}}}&\multicolumn{2}{c}{\smash{\raisebox{.5\normalbaselineskip}{$\cdots$}}}&\multicolumn{2}{c}{\smash{\raisebox{.5\normalbaselineskip}{$R_l^{nn}$}}}&0\\

0&0&\multicolumn{2}{c}{{\cdots}}&0&0&1
\end{array}\right),
\end{equation}where each $R^{ij}_l$ is a 2-by-2 block of the form $\left(\begin{array}{cc}x&y\\-y&x\end{array}\right)$, with $x, y \in \rrr$. Then $\prod_{l=1}^g[M_{2l-1},M_{2l}]$ cannot be a generic torus element.  \end{lem}

\begin{proof}Consider the map 
\begin{align*}
\kappa:U(n) &\hookrightarrow SO(2n+1)\\
\kappa: \left(\begin{array}{ccc}a_{11}&\cdots & a_{1n}\\ \vdots & \ddots & \vdots \\ a_{n1}& \cdots & a_{nn} \end{array}\right) & \mapsto \left(\begin{array}{cccccc} \Re a_{11} & -\Im a_{11} & \cdots &&&0\\ \Im a_{11}& \Re a_{11} &&&&0\\ &&\ddots &&&\vdots \\ &&&\Re a_{nn} & -\Im a_{nn} &0 \\ &&& \Im a_{nn} & \Re a_{nn}&0\\0&0&\cdots&0&0&1\end{array}\right).
\end{align*} This is an injective homomorphism, and $\kappa(M) \in T \iff M$ is diagonal in $U(n)$. Under the hypothesis of this lemma, each $M_l$ is in the image of $\kappa$; say $M_l = \kappa(N_l)$, for $N_l \in U(n)$. Then 
\[\prod_{l=1}^g[M_{2l-1},M_{2l}] = \prod_{l=1}^g[\kappa(N_{2l-1}), \kappa(N_{2l})]=\kappa\left(\prod_{l=1}^g[N_{2l-1},N_{2l}]\right).
\] Observe that if $N=\prod_{l=1}^g[N_{2l-1},N_{2l}]$ has determinant 1. If $N$ is not diagonal, then $\kappa(N)$ is not a torus element (so certainly not a generic torus element). Assume $\kappa(N)$ is a torus element. Then $N$ is a diagonal matrix in $U(n)$ of determinant 1, so is $\left(\begin{array}{ccc}e^{i\theta_1}&&\\&\ddots&\\&&e^{i\theta_n}\end{array}\right)$, where $\theta_1 + \ldots + \theta_n \in 2\pi \zzz$. Thus $\kappa(N)$ is $(\theta_1, \ldots, \theta_n) \in T$, which is not generic in the sense of Definition \ref{generic} because $\theta_1 + \ldots + \theta_n \in 2 \pi \zzz$.  
\end{proof}

\begin{cor}
The sections $\{\spijx \vert 1 \leq i, j \leq n, x \in \{a_1, \ldots, a_g, b_1, \ldots, b_g\}\}$ of the line bundles $\lpij$ have no common zeros. \end{cor}

\begin{proof}Fix $x \in \{a_1, \ldots, a_g, b_1, \ldots, b_g\}$ and consider the sections $\{\spijx \mid 1 \leq i, j, \leq n \}$. These sections all vanish when 
\[\rho(x)= \left( \begin{array}{cccccc} x_{11} & -y_{11} &  \cdots & x_{1n} & -y_{1n} & z_1\\ y_{11} & x_{11} &  \cdots & y_{1n} &x_{1n}& z_2\\ \vdots & \vdots & \ddots & \vdots & \vdots & \vdots \\x_{n1} & -y_{n1} & \cdots & x_{nn} & -y_{nn} & z_{2n-1} \\ y_{n1} & x_{n1} & \cdots & y_{nn} & x_{nn} & z_{2n} \\ w_1 & w_2 & \cdots & \cdots & w_{2n} & u \end{array}\right).
\] Since $\rho(x) \in SO(2n+1)$, we must have $z^2_{2l-1}=z^2_{2l}$, and $z_{2l-1}z_{2l}=0$, so $z_1 = \cdots = z_{2n} =0$. Similarly $w_1 = \cdots = w_{2n} = 0$, hence $u=1$, so $\rho(x)$ is in the form \eqref{plusmat} from Lemma \ref{genplus}. At points in $S_{n,g}(t)$ where all of our sections vanish, therefore, $\rho(x)$ is in this form for every $x \in \setgens$, so by Lemma \ref{genplus}, $\prod_{l=1}^g[\rho(a_l),\rho(b_l)]$ cannot be a generic torus element. However, $\prod_{l=1}^g[\rho(a_l),\rho(b_l)] = \rho(c) = t$ which was chosen to be a generic torus element. Thus there are no such points in $S_{n,g}(t)$, that is, the locus on which every one of these sections vanishes is empty. \end{proof}

In the coming discussion, we will make use of the following definition from our earlier paper \cite{aj}.

\begin{defn} 
\label{block}
Let $X$ be a finite set. A \emph{block} $B$ in $X\times X$ is a subset of $X\times X$ of the form $V \times V^c$, where $\varnothing \subsetneq V \subsetneq X$ is a proper nonempty subset of $X$. We denote the set of all blocks in $X \times X$ by $\mathcal{B}[X]$. If $B=V \times V^c$, let $\bar{B}=V^c\times V$. We will also make use of the indicator functions 
\[
\epsilon_B(i,j) := \begin{cases}
1 & (i,j) \in B\sqcup\bar{B}\\
-1 & \text{otherwise}
\end{cases}
\]
and \[
\epsilon_V(i) := \begin{cases}
1 & i \in V\\
-1 & \text{otherwise}
\end{cases}.
\]
\end{defn}

\begin{notn}
For a positive integer $m$, we will denote by $[m]$ the set $\{1, \ldots, m\}$. 
\end{notn}

For $l \in [n]$, let $E_l \in O(2n+1)$ be the matrix 
                
\bordermatrix{&&&&&2l&&\cr
				&1&&&&&\cr
				&&\ddots&&&&\cr
				&&&1&&\cr
				&&&&0&1&&&\cr
				2l&&&&1&0&\cr
				&&&&&&1\cr
				&&&&&&&\ddots\cr
				&&&&&&&&1\cr}                
(so conjugation by $E_l$ switches the $(2l-1)^{th}$ and $(2l)^{th}$ rows, and the $(2l-1)^{th}$ and $(2l)^{th}$ columns).

\begin{lem} 
\label{genminus} 
Let $B=V \times V^c \in \mathcal{B}[[n]]$. Suppose $M_l \in SO(2n+1)$ for each $1 \leq l \leq 2g$, and each $M_l$ has the form 
\begin{equation} \label{minusmat}
\left(\begin{array}{ccccccc}
&&&&&&0\\
\multicolumn{2}{c}{\smash{\raisebox{.5\normalbaselineskip}{$S^{11}_l$}}}&\multicolumn{2}{c}{\smash{\raisebox{.5\normalbaselineskip}{$\cdots$}}}&\multicolumn{2}{c}{\smash{\raisebox{.5\normalbaselineskip}{$S^{1n}_l$}}}&0\\
\multicolumn{2}{c}{\vdots}&\multicolumn{2}{c}{\ddots}&\multicolumn{2}{c}{\vdots}&\vdots\\
&&&&&&0\\
\multicolumn{2}{c}{\smash{\raisebox{.5\normalbaselineskip}{$S^{n1}_l$}}}&\multicolumn{2}{c}{\smash{\raisebox{.5\normalbaselineskip}{$\cdots$}}}&\multicolumn{2}{c}{\smash{\raisebox{.5\normalbaselineskip}{$S_l^{nn}$}}}&0\\

0&0&\multicolumn{2}{c}{{\cdots}}&0&0&1
\end{array}\right),
\end{equation} where $S^{ij}_l$ takes the form 
\[
\begin{cases} \left(\begin{array}{cc}x&-y \\ y & x \end{array}\right) & (i,j) \notin B \sqcup \bar{B}\\ \left( \begin{array}{cc} w& z \\ z & -w \end{array}\right) & \text{otherwise}.
\end{cases}
\] 
Then $\prod_{l=1}^g[M_{2l-1},M_{2l}]$ cannot be a generic torus element. 
\end{lem}	

\begin{proof} Let $E= \prod_{k \in V} E_k$, and observe that if $M_l$ is in the form \eqref{minusmat} then $EME^{-1}$ is in the form \eqref{plusmat}. Thus $E\prod_{l=1}^g[M_{2l-1},M_{2l}]E^{-1}$ is not a generic torus element. Notice that if $h = ( \theta_1, \ldots, \theta_n)\in T$ is a generic torus element, then $EhE^{-1}= (\epsilon_V(1)\theta_1, \ldots, \epsilon_V(n) \theta_n)$ is also generic. Hence $\prod_{l=1}^g[M_{2l-1}, M_{2l}]$ is not a generic torus element.  
\end{proof}		

\begin{cor} Let $B \in \mathcal{B}[[n]]$. The sections 
\begin{equation*}
\{s^{-\epsilon_B(i,j)}_{ij}(x) \vert 1 \leq i, j \leq n, x = \genss\}
\end{equation*}have no common zeros. \end{cor}	
\begin{proof}Consider the sections $\{s^{-\epsilon_B(i,j)}_{ij}(x) \vert 1 \leq i, j \leq n, x = \genss\}$. These sections all vanish when each
\[
\rho(x) = \left(\begin{array}{ccccccc}
&&&&&&z_1\\
\multicolumn{2}{c}{\smash{\raisebox{.5\normalbaselineskip}{$R^{11}_x$}}}&\multicolumn{2}{c}{\smash{\raisebox{.5\normalbaselineskip}{$\cdots$}}}&\multicolumn{2}{c}{\smash{\raisebox{.5\normalbaselineskip}{$R^{1n}_x$}}}&z_2\\
\multicolumn{2}{c}{\vdots}&\multicolumn{2}{c}{\ddots}&\multicolumn{2}{c}{\vdots}&\vdots\\
&&&&&&z_{2n-1}\\
\multicolumn{2}{c}{\smash{\raisebox{.5\normalbaselineskip}{$R^{n1}_x$}}}&\multicolumn{2}{c}{\smash{\raisebox{.5\normalbaselineskip}{$\cdots$}}}&\multicolumn{2}{c}{\smash{\raisebox{.5\normalbaselineskip}{$R_x^{nn}$}}}&z_{2n}\\

w_1&w_2&\multicolumn{2}{c}{{\cdots}}&w_{2n-1}&w_{2n}&u
\end{array}\right),
\] where $R^{ij}_x$ takes the form 
\[\begin{cases} \left(\begin{array}{cc}x& y \\ -y & x \end{array} \right) & \epsilon_B(i,j)=-1 \\ \left( \begin{array}{cc}z&w \\ w & -z \end{array} \right) & \epsilon_B(i,j) = 1
\end{cases}.
\] Since $\rho(x) \in SO(2n+1)$, this forces $z_1 = \cdots = z_{2n} = w_1 = \cdots = w_{2n} = 0$, and thus $u = 1$. So the sections we are considering all vanish when every $\rho(x)$ is of the form \eqref{minusmat} described in Lemma \ref{genminus}, in which case $\prod_{l=1}^g[\rho(a_l), \rho(b_l)]$ cannot be a generic torus element. But again, $\prod_{l=1}^g[\rho(a_l), \rho(b_l)]= \rho(c) = t$ was chosen to be generic, so these sections have no common zeros. 
\end{proof}

Now fix $x \in \setgens$, and let $C = U \times U^c \in \mathcal{B}[[n]]$ be another block (not necessarily distinct from $B$). Consider the sections 
\begin{equation}\label{msections}\{\spijx \vert \epsilon_B(i,j) = \epsilon_C(i,j) = -1 \} \cup \{\smijx \mid \epsilon_B(i,j) = 1, \epsilon_c(i,j)=-1\} \cup \{ \spijx, \smijx \mid (i,j) \in C\}.
\end{equation} If these sections all vanish, then 
\[
\rho(x) = \left(\begin{array}{ccccccc}
&&&&&&z_1\\
\multicolumn{2}{c}{\smash{\raisebox{.5\normalbaselineskip}{$T^{11}_x$}}}&\multicolumn{2}{c}{\smash{\raisebox{.5\normalbaselineskip}{$\cdots$}}}&\multicolumn{2}{c}{\smash{\raisebox{.5\normalbaselineskip}{$T^{1n}_x$}}}&z_2\\
\multicolumn{2}{c}{\vdots}&\multicolumn{2}{c}{\ddots}&\multicolumn{2}{c}{\vdots}&\vdots\\
&&&&&&z_{2n-1}\\
\multicolumn{2}{c}{\smash{\raisebox{.5\normalbaselineskip}{$T^{n1}_x$}}}&\multicolumn{2}{c}{\smash{\raisebox{.5\normalbaselineskip}{$\cdots$}}}&\multicolumn{2}{c}{\smash{\raisebox{.5\normalbaselineskip}{$T_x^{nn}$}}}&z_{2n}\\

w_1&w_2&\multicolumn{2}{c}{{\cdots}}&w_{2n-1}&w_{2n}&u
\end{array}\right),
\] where $T^{ij}_x$ takes the form
\[\begin{cases}  \left( \begin{array}{cc}x&-y \\ y & x \end{array}\right) & \epsilon_B(i,j)=\epsilon_C(i,j)=-1 \\ \left( \begin{array}{cc}z&w \\ w & -z \end{array}\right) & \epsilon_B(i,j)=1 \text{ and } \epsilon_C(i,j) = 1 \\ \left( \begin{array}{cc}0&0 \\ 0 & 0\end{array}\right) & \epsilon_C(i,j)=1
\end{cases}.
\] Since $\rho(x) \in SO(2n+1)$, this forces $z_{2i-1}=z_{2i}=0$ for all $i \in U$; thus $\rho(x)$ is block diagonal up to reordering of basis elements, so also $T^{ij}_x= \left( \begin{array}{cc}0&0 \\ 0 & 0\end{array}\right)$ for $(i,j)  \in \bar{C}$. Thus again, $z_1 = \cdots = z_{2n} = w_1 = \cdots = w_{2n} = 0$, and $u=1$. Observe that in particular, $\rho(x)$ is in the form \eqref{minusmat} described in Lemma \ref{genminus}, and further, that this form \eqref{minusmat} is independent of $C$. More generally, $\rho(x)$ takes the same form \eqref{minusmat} when the sections obtained from those in \eqref{msections} by replacing the block $C$ with a union of blocks all vanish. If we instead take $B$ to be the empty set, then $\rho(x)$ is in the form \eqref{plusmat} described in Lemma \ref{genplus}. This leads to the following lemma:

\begin{lem}
\label{withunions}
Let $C \in \mathcal{B}[[n]]\cup\{\varnothing \}$. For $ 1 \leq l \leq 2g$, let $D_l$ be a (possibly empty) union of blocks in $\mathcal{B}[[n]]$. Then the following collection of sections has no common zeros:
\begin{multline}\label{unionsections}
\bigcup_{l=1}^g \{s^+_{ij}(a_l), s^-_{ij}(a_l) \mid (i,j) \in D_l\} \cup \{s^{-\epsilon_C(i,j)}_{ij}(a_l) \mid (i,j) \notin D_l \cup \bar{D_l} \} \\ \cup \bigcup_{l=g+1}^{2g} \{s^+_{ij}(b_{l-g}),s^-_{ij}(b_{L_g}) \mid (i,j) \in D_l \} \cup \{s^{-\epsilon_C(i,j)}_{ij}(b_{l-g}) \mid (i,j) \notin D_l \cup \bar{D_l} \}.
\end{multline}
\end{lem}

\begin{proof} As discussed above, if these sections all vanish then $\rho(x)$ has the form \eqref{minusmat} for every $x \in \setgens$. Thus $\prod_{l=1}^g[\rho(a_l), \rho(b_l)]$ cannot be a generic torus element by Lemma \ref{genminus}. But $\prod_{l=1}^g[\rho(a_l), \rho(b_l)]=\rho(c) = t$ was chosen to be a generic torus element, so these sections have no common zeros.
\end{proof}
\begin{rmk} \label{blockrmk}
In Lemma \ref{genminus}, we proved that if each $\rho(x)$ is of the form \eqref{minusmat}, then $\prod_{l=1}^g[\rho(a_l),\rho(b_l)]$ cannot be a generic torus element. Observe that if we have $2g$ block diagonal matrices $A_1, \ldots, A_g, B_1, \ldots, B_g$, where the first block in each is a $k$-by-$k$ matrix of the form \eqref{minusmat} 
$\left(\begin{array}{c|c}\eqref{minusmat}& 0 \\ \hline 0 & *\end{array}\right)$, then their product of commutators $\prod_{l=1}^g[A_l, B_l]$ also cannot be a generic torus element.
\end{rmk}

We will use this idea to find collections of sections with no common zeros by extending collections that worked for lower-rank cases. Note that in the example above, the matrices only need to be block diagonal up to reordering of the basis elements. Thus it will be useful to introduce the following notation:
\begin{notn} Let $G=SO(2n+1)$, let $X$ be a nonempty finite subset of $[n]$, let $B \in \mathcal{B}[X] \cup \{\varnothing\}$, and let $D$ be a $2g$-tuple of unions of blocks in $\mathcal{B}[X]$. Let $x_l= \begin{cases} a_l & 1 \leq l \leq g \\ b_{l-g} & g+1 \leq l \leq 2g \end{cases}$. Define the collection of sections 
\[Q_n(X,B,D):=\bigcup_{l=1}^{2g} \{s^+_{ij}(x_l),s^-_{ij}(x_l) \mid (i,j) \in D_l \} \cup \{s^{-\epsilon_B(i,j)}_{ij}(x_l) \mid (i,j) \notin D_l \cup \bar{D_l} \}.
\] \end{notn}
When the sections in $Q_n(X,B,D)$ all vanish, the $\vert X \vert$-by-$\vert X \vert$ submatrix of each $\rho(x)$ induced by considering only the rows and columns $2i-1$ and $2i$ for elements $i$ in $X$ must take the form \eqref{minusmat}. Lemma \ref{withunions} says that the sections in $Q_n([n], B, D)$ have no common zeros. 

\begin{defn} 
\label{recursiveAsets}
Consider sets $\mathcal{A}_n$, for $n \in \nnn$, defined recursively as follows: \begin{itemize} \item if $B\in \mathcal{B}[[n]]\cup \{\varnothing\}$ and $0 \leq r \leq 2g$, then $Q_n([n],B,D) \in \mathcal{A}_n$. \item if $X \sqcup Y = [n]$ is a partition, $\psi: [ \vert X \vert ] \to X$ is a bijection, and $P \in \mathcal{A}_{\vert X \vert}$, then $\psi_*P \cup \cup_{i \in X, j \in Y} \cup_{x \in \setgens} \spijx \smijx \in \mathcal{A}_n$.  \end{itemize}
\end{defn}

This induced map $\psi_*$ sends a section $s^\pm_{ij}(x)$ of the line bundle $\lpmij$ over $S_{\vert X \vert,g}(t)$ to the section $s^\pm_{\psi(i),\psi(j)}(x)$ of the line bundle $L^\pm_{\psi(i),\psi(j)}$ over $S_{n,g}(t)$.
\begin{prop} If $P \in \mathcal{A}_n$ is a collection of sections in the set just described, then the sections in $P$ have no common zeros. 
\end{prop}
\begin{proof}An element $P \in \mathcal{A}_n$ takes the form 
\[P = Q_n(X, B, D) \cup \bigcup_{k=1}^d \bigcup_{\substack{i \in X \cup Y_1 \cup \cdots \cup Y_{k-1}\\ j \in Y_k}}\bigcup_{x \in \setgens}\spijx \smijx,
\] for some $0 \leq k \leq n$. Observe that this collection contains the collection of sections 
\[Q_n(X,B,D) \cup \bigcup_{\substack{i \in X \\ j \in Y_1 \cup \cdots \cup y_k}}\bigcup_{x \in \setgens}\spijx \smijx.
\]
When these sections all vanish, each $\rho(x)$ is (up to reordering) of the form 
\[\left(\begin{array}{ccc|ccccccc|c}
&&&&&&&&&&z_1\\
&A_x&&&&&0&&&&\vdots\\
&&&&&&&&&&z_{2\vert X \vert}\\
\hline 
&&&&&&&&&&\\
&&&&&&&&&&\\
&*&&&&&*&&&&*\\
&&&&&&&&&&\\
&&&&&&&&&&\\
\end{array}\right),
\] where each $A_x$ is of the form \eqref{minusmat}. But the form of $A_x$ forces $z_{2l-1}=z_{2l}=0$ for each $l \in X$, so each $\rho(x)$ is in fact block diagonal (up to reordering), of the form discussed in Remark \ref{blockrmk}. Hence these sections have no common zeros. 
\end{proof}
\begin{cor} \label{goodthingsvanish}Suppose 
\[ \bigcup_{\substack{1 \leq i,j \leq n \\ x \in \setgens}}(\spijx)^{k^+_{ij}(x)}(\smijx)^{k^-_{ij}(x)} \in \mathcal{A}_n,
\]where $k^+_{ij}(x)$ and $k^-_{ij}(x) \in \nnn$ for $1 \leq i, j \leq n$ and $x \in \setgens$, and $k^-_{ij}(x)=0$ whenever $i=j$. Then the cohomology class
\[\prod_{1 \leq i, j \leq n}c_1(\lpij)^{\sum_x k^+_{ij}(x)}c_1(\lmij)^{\sum_xk^-_{ij}(x)}
\] vanishes in $H^*(S_{n,g}(t);\qqq)$. \qed \end{cor}
\section{The combinatorics} \label{heart}
So far we have found a class of ``good'' products of the $c_1(\lpmij)$ which vanish. The rest of this paper is devoted to proving that any product of at least $2gn^2 + \frac{1}{2}(n-1)(n-2)$ of the $c_1(\lpmij)$ is equivalent in $H^*(S_{n,g}(t);\qqq)$ to a combination of these ``good'' products, and hence also vanishes. 

The proof is combinatorial, and makes extensive use of the relations
\begin{align} \label{relns}
c^+_{ij} &= c^+_{ji} \\ 
c^-_{ij} &=-c^-_{ji} \\
c^+_{ij} -c^+_{jk}+c^-_{ki}&=0.
\end{align}
For the terms appearing in these relations to be defined, we must have fixed the rank $n$, but the same relations hold no matter which $n$ is chosen. In order to use inductive arguments we wish to be able to make statements that don't rely on having fixed the rank. We thus move away from considering the Chern classes themselves and instead study the combinatorial properties of another ring $R$ whose elements satisfy the same relations \eqref{relns} as our Chern classes. 

\begin{defn} 
Let $X$ be a finite subset of $\nnn$. Define the associated auxiliary sets \begin{itemize} 
\item $Y^+(X) := \{y^+_{ij} \mid i,j \in X\}$ \item $Y^-(X):= \{ y^-_{ij} \mid i, j \in X; i \neq j\}$ \item $Y(X):= Y^+(X) \sqcup Y^-(X)$,\end{itemize} as well as subsets $Y_z(X):=\{y^\pm_{ij}\in Y(X) \mid i = z \text{ or } j = z\}$ for each $x \in X$. 
If $B \in \mathcal{B}[X]$ is a block as defined in Definition \ref{block}, then we also introduce 
\begin{itemize} \item $Y^+_B(X):=\{y^{-\epsilon_B(i,j)}_{ij} \mid i, j\in X\}$ \item $Y^-_B(X):= \{y^{\epsilon_B(i,j)}_{ij} \mid i, j \in X ; i \neq j\}$
\end{itemize}
\end{defn}

We would like to form a quotient of $\qqq[Y(X)]$ by relations corresponding to those satisfied by the $c_1(\lpmij)$ listed above \eqref{relns}. 

\begin{defn}
Let $I \subset \qqq[Y(X)]$ be the ideal generated by the elements
\begin{itemize}
\item $\ymij + y^-_{ji}$ \item $\ypij- y^+_{ji}$ \item $\ymij + y^-_{jk} + y^-_{ki}$ \item $\ypij -y^+_{jk} + y^-_{ki}$
\end{itemize} for all triples of elements $i,j,k \in X$. Let $R:= \qqq[Y(X)]/I$ be the quotient of $\qyx$ by this ideal; if $p \in \qyx$ we will denote by $[p]$ its image in $R$.  Note that this quotient preserves the grading by degree of $\qyx$. 
\end{defn}

In the previous section we found sets $\mathcal{A}_n$ of collections of sections with no common zeros, and concluded that the corresponding products of Chern classes vanished in $H^*(S_{n,g}(t))$. We wish now to work in the ring $R$ and not in $H^*(S_{n,g}(t))$, so we introduce the map 
\begin{align*}
\alpha: \cup_n \{\text{sections } \spmijx \text{ of }\lpmij \to S_{n,g} \} & \to \qyx \\ \spmijx & \mapsto \ypmij.
\end{align*}
Note that $\alpha$ is a map of sets, not a ring homomorphism; we will use the same notation for the map that takes a set of sections to the product of their images in $\qyx$. Observe that $\alpha$ forgets both $n$ and $x$. 
The main goal of the remainder of this paper is to prove that for any monomial $q \in \qyx$ of degree at least $2gn^2+\frac{1}{2}(n-1)(n-2)$, there exist some $q_i \in \mathcal{A}_n$ and monomials $\theta_i \in \qyx$ such that $[q]=[\sum_i \theta_i \alpha(q_i)]$ in $R$. Recall that the sets $\mathcal{A}_n$ were defined recursively. We make this more explicit in the statement of the proposition:

\begin{prop} 
Let $X \subset \nnn$ with $\vert X \vert = n$. Let $p \in \qyx$ be a homogeneous polynomial of degree at least $2gn^2+\frac{1}{2}(n-1)(n-2)$. Then for each $B = V \times V^c \in \mathcal{B}[X]$ and $C \in \mathcal{B}[X]\cup \{\varnothing\}$, we can find:
\begin{itemize} \item a finite set $\mathcal{D}$ whose elements are $2g$-tuples of unions of blocks in $\mathcal{B}[X]$ \item homogeneous polynomials $\theta_B, \phi_B, \psi_C \in \qyx$, as well as polynomials $\chi_D$ for each $D \in \mathcal{D}$ \item elements $p_V \in \mathcal{A}_{\vert V \vert}$ and $p_{V^c} \in \mathcal{A}_{\vert V^c \vert}$, and \item bijections $f_V:[\vert V \vert] \to V$ and $f_{V^c}:[\vert V^c \vert ] \to V^c$ such that\end{itemize}
\begin{multline*}[p] = [\sum_{B=V\times V^c \in \mathcal{B}[X]}\prod_{(i,j) \in B}(\ypij \ymij)^{2g}(\theta_B\cdot\alpha((f_V)_*(p_V))+\phi_B\cdot \alpha((f_{V^c})_*(p_{V^c})))  \\ + \sum_{C \in \mathcal{B}[X] \cup \{\varnothing\}}\psi_C \sum_{D \in \mathcal{D}}\chi_D \prod_{l=1}^{2g}\left(\prod_{(i,j) \in D_l}\ypij \ymij\right)\left(\prod_{(i,j) \notin D_l \cup \bar{D_l}}y^{-\epsilon_C(i,j)}_{ij}\right)]
\end{multline*}
\end{prop}

There are several points during the course of the proof of this proposition when we apply the pigeonhole principle, generally in order to show that the restriction of a polynomial to some subring of $\qyx$ has high enough degree to apply an inductive hypothesis. In order to reduce clutter we collect the relevant calculations into the following five lemmas. 

\begin{lem} \label{oddsomez} Let $X$ be a finite subset of $\nnn$ with $\vert X \vert = m > 3$, and let $p \in \qyx$ be a monomial of degree at least $2gm(m-1)-m+1$. Then there exists some $z \in X$ such that if we factorise $p$ as $p=q_zr_z$, where $q_z \in \qyxmz$ and $r_z \in \qyzx$ are monomials, then $q_z$ has degree at least $2g(m-1)(m-2)-m+2$. 
\end{lem}

\begin{proof} Write 
\[p=\lambda \prod_{i,j \in X }(\ypij)^{d^+_{ij}}(\ymij)^{d^-_{ij}}\]
(where $d^-_{ii}=0$ for all $i$). Given $z \in X$, 
\[q_z = \prod_{i,j \in X \setminus \{z\} }(\ypij)^{d^+_{ij}}(\ymij)^{d^-_{ij}}.
\]
Note that each factor $\ypij$ of $p$ appears in $q_z$ precisely when $i, j \neq z$, and thus appears in at least $m-2$ of the $q_z$ as $z$ ranges over $X$ (exactly $m-2$ except for the $y^+_{ii}$ which appear in $m-1$). Hence $\prod_{z=1}^mq_z=p^{m-2}\prod_{i \in X \setminus \{z\}}(y^+_{ii})^{d^+_{ii}}$ has degree $\geq (m-2)(2gm(m-1)-m+1)$. Suppose by way of contradiction that each $q_z$ has degree at most $2g(m-1)(m-2)-m+1$. Then $\prod_{z=1}^mq_z$ has degree at most 
\begin{align*} m(2g(m-1)(m-2)-m+1) & = 2gm(m-1)(m-2)-m+1\\ &< (m-2) 2gm(m-1)-(m-2)(m-1) \\&=(m-2)(2g(m-1)-m+1).
\end{align*} This is a contradiction, so the desired $z$ must exist. 
\end{proof}

\begin{lem} \label{indind}
Let $X = \sehfwz$ be a subset of $\nnn$ with $\vert X \vert = m$. Let $p \in \qyzx$ be a monomial of degree at least $2gm(m-1)-m+1-4gwh$ that factorises as $p = p_h p_w$, where $p_h \in \qyehz$ and $p_w \in \qyfwz$ are monomials. Then either $\deg p_h \geq 2gh(h+1)-h$, or $\deg p_w \geq 2gw(w+1)-w $. 
\end{lem}
\begin{proof}Suppose $\deg p_h \leq 2gh(h+1)-h-1$. Then 
\begin{align*} \deg p_w &\geq 2gm(m-1)-m+1-4gwh-2gh(h+1)+h+1 \\ &= 2gm(w+h)-(w+h)-4gwh -2gh(m-w)+(m-w) \\ &=2gw(m+h)-4gwh-w-h+m-w \\&=2gw(w+h+1+h)-4gwh-w+1 \\&= 2gw(w+1)-w+1 > 2gw(w+1)-w.\qedhere
\end{align*} 
\end{proof}

\begin{lem} \label{bozsqind} Let $X = H \sqcup W$ be a finite subset of $\nnn$ with $\vert H \vert = h>0$, $\vert W \vert = w>0$, and $\vert X \vert = w + h = n$. Let $p \in \qyx$ be a monomial of degree at least $2gn^2+ \frac{1}{2}(n-1)(n-2)-4gwh$ that factorises as $p = p_w p_h$, where $p_w \in \qqq[Y(W)]$ and $p_h \in \qqq[Y(H)]$. Then either $\deg p_w \geq 2gw^2+ \frac{1}{2}(w-1)(w-2)$ or $\deg p_h \geq 2gh^2+ \frac{1}{2}(h-1)(h-2)$. 
\end{lem}
\begin{proof} Suppose $\deg p_h \leq 2gh^2 + \frac{1}{2}(h-1)(h-2)-1$. Then 
\begin{align*}\deg p_w &\geq 2gn^2+ \frac{1}{2}(n-1)(n-2)-4gwh-2gh^2-\frac{1}{2}(h-1)(h-2) +1 \\ &= 2g(h^2+2wh + w^2) + \frac{1}{2}(w+h-1)(w+h-2)-4gwh-2gh^2-\frac{1}{2}(h-1)(h-2)+1 \\ &=2gw^2+ \frac{1}{2}(h-1)(h-2)+ \frac{1}{2}w(h-2)+ \frac{1}{2}w(h-1)+\frac{w^2}{2}-\frac{1}{2}(h-1)(h-2)+1 \\
&= 2gw^2+ \frac{1}{2}(w^2-2w-w+2)+wh \\ &= 2gw^2 + \frac{1}{2} (w-1)(w-2)+wh > 2gw^2+ \frac{1}{2}(w-1)(w-2). \qedhere
\end{align*}
\end{proof}

\begin{lem} \label{2goroct} Let $p=qr$ be a monomial of degree at least $2gn^2+ \frac{1}{2}(n-1)(n-2)-n(n-1)g-2g(n-1)$. Then either $\deg q \geq 2g$ or $\deg r \geq n(n-1)g-n+2$. 
\end{lem}
\begin{proof} Suppose $\deg r \leq n(n-1)g-n+1$. Then 
\begin{align*} \deg q & \geq 2gn^2 + \frac{1}{2}(n-1)(n-2)-n(n-1)g-2g(n-1)-n(n-1)g+n-1 \\ &=2gn^2 -2gn(n-1)-2g(n-1)+ \frac{1}{2}(n-1)(n-2)+n-1 \\ &=2g + \frac{1}{2}n(n-1) \geq 2g. \qedhere
\end{align*}
\end{proof}

\begin{lem} \label{indoroct} Let $w, h \in \nnn$ with $w+h=m-1$, and let $p=qr$ be a monomial of degree at least $2gm(m-1)-m+1-4gwh-h(h+1)g$. Then either $\deg q \geq 2gw(w+1)-w$ or $\deg r \geq h(h+1)g-(h+1)+2$. 
\end{lem}
\begin{proof} Suppose $\deg r \leq h(h+1)g-(h+1)+1$. Then \begin{align*} \deg q & \geq 2gm(m-1)-m+1-4gwh-h(h+1)g-h(h+1)g+(h+1)-1 \\ &= 2g(w+h+1)(w+h)-(w+h)-4gwh-2gh(h+1)+h \\ &= 2gw^2 + 2gh(h+1) + 2gw -w -2gh(h+1) \\ &= 2gw(w+1)-w . \qedhere
\end{align*}
\end{proof}

We will also need the following results about interactions between different blocks. 

\begin{lem} \label{cd1sbextends} Suppose $X = \{e_1, \ldots, e_h, f_1, \ldots, f_w, z\}$. Let $B \in \bbb[X\setminus \{z\}]$ be the block $B = \seh \times \sfw$, and let $C$ be a block in $\bbb[\{\ffw,z\}]$. Then the union $B \cup C \cup \bar{C}$ contains a block $D \in \bx$. 
\end{lem}

\begin{proof} Suppose without loss of generality $C = \{f_1, \ldots, f_d\} \times \{f_{d+1}, \ldots, f_w, z\}$, for some $1 \leq d \leq w$. Then $D = \{ \eeh, f_{d+1}, \ldots, f_w, z\} \times \{f_1, \ldots, f_d\} \subset B \cup C \cup \bar{C}$.
\end{proof}
\begin{lem} \label{ebfbextendinb} Again, suppose $X = \{\eeh, \ffw, z\}$, and let $B=\seh \times \sfw \in \bxmz$. Let $C \in \bbb[\{\eeh, z\}]$ and $E \in \bbb[\{\ffw, z\}]$. Then there exists a block $A_{CE} \in \bx$ such that either $C\cup E \subset A_{CE} \subset C \cup E \cup B \cup \bar{B}$, or $C \cup \bar{E} \subset A_{CE} \subset C \cup \bar{E} \cup B \cup \bar{B}$. 
\end{lem}
\begin{proof}Suppose without loss of generality $C = \{e_1, \ldots, e_k\} \times \{e_{k+1}, \ldots, e_h, z\}$, for some $1 \leq k \leq h$. Either $E$ or $\bar{E}$ has the form $\{f_1, \ldots, f_l \} \times \{f_{l+1}, \ldots, f_w, z\}$ (up to relabelling), so take $A_{CE}= \{ e_1, \ldots, e_k, f_1, \ldots, f_l\} \times \{e_{k+1}, \ldots, e_h, f_{l+1}, \ldots, f_w, z\}$.
\end{proof}

\begin{lem}[The symmetric difference of two blocks is a block] \label{symdiff} Let $B, C \in \bx$. Then there exists a block $D \in \bx$ such that $(B\cup \bar{B}) \triangle (C \cup \bar{C})=D \cup \bar{D}$. 
\end{lem}
\begin{proof}
Suppose $B  = \seh \times \sfw$, and 
\[
C = \{e_1, \ldots, e_s, f_1, \ldots, f_t\} \times \{e_{s+1}, \ldots, e_h, f_{t+1}, \ldots, f_w\}.
\] Then take $D = \{e_1, \ldots, e_s, f_{t+1}, \ldots, f_w\} \times \{e_{s+1}, \ldots, e_h, f_1, \ldots, f_t\}$.
\end{proof}

\begin{rmk} \label{restrictblock} Let $B \in \bx$ and $z \in X$. Then $B \vert_{\xmz}:=B \cap ((X \setminus \{z\}) \times (X \setminus \{z\})) \in \bxmz$. \end{rmk}

We are now ready to study the quotient ring $R$. We begin by observing that the polynomial ring obtained by adjoining the elements in the subset $Y^-(X)\subset Y(X)$ to $\qqq$ is isomorphic to the polynomial ring introduced in \cite{aj} to study the case $G=SU(n)$; we can therefore use the following result from our earlier work:

\begin{lem} \label{propfromaj} Let $X$ be a finite set. Let $p \in \qymx$ be a monomial of degree at least $\frac{1}{2}\vert X \vert(\vert X \vert -1)a-\vert X \vert +2$, for some $a \in \nnn$. Then for each $B \in \bx$ we can find a monomial $\phi_B \in \qymx$ such that 
\[[p] = \left[\sum_{B \in \bx}\phi_B\prod_{(i,j) \in B}(\ymij)^a\right].
\]
\end{lem}
\begin{proof} This follows from replacing $2g$ by $a$ in the proof of Proposition 3.6 in \cite{aj}. \end{proof}

\begin{rmk} \label{isomflip}Let $X$ be a finite set and let $B=V \times V^c$ be a block in $\bx$. Consider the isomorphism $f_B:\qyx \to \qyx$ given by 
\[f_B:\ypmij \mapsto \begin{cases} - \ypmij & i, j \in V \\ \ypmij & i, j \notin V \\ -y^\mp_{ij} & (i,j) \in B \\ y^\mp_{ij} & (i,j) \in \bar{B} \end{cases}
\]or more concisely, $y^{\pm}_{ij} \mapsto \mp\epsilon_V(i)y^{-\epsilon_B(i,j)}_{ij}$. This map restricts to an isomorphism $f_B:\qymx \to \qymbx$; it is straightfoward to check that $f_B$ descends to an isomorphism on the quotient $R$.\end{rmk}

\begin{cor} \label{ajpropflipblock} Let $X$ be a finite set, let $B=V\times V^c \in \bx$ be a block, let $a \in \nnn$, and let $p \in \qymbx$ be a monomial of degree at least $\frac{1}{2}\vert X \vert (\vert X \vert -1)a - \vert X \vert +2$. Then for each $C \in \bx$ we can find a monomial $\phi_C \in \qymbx$ such that
\[[p]= \left[  \sum_{C \in \bx}\phi_C \prod_{(i,j) \in C}(y_{ij}^{\epsilon_B(i,j)})^a\right].
\]
\end{cor}
\begin{proof}
Consider the isomorphism $f_B:\qymx \to \qymbx$ defined above. Applying Lemma \ref{propfromaj} to $f_B^{-1}(p)$, we can write 
\[[f^{-1}_B(p)]=\left[\sum_{C \in \bx}\phi_C\prod_{(i,j) \in C}(\ymij)^a\right].
\]
Thus 
\[ [p]= \left[ \sum_{C \in \bx}f_B(\phi_C) \prod_{(i,j) \in C}(-1)^{a\lambda_C}(y_{ij}^{\epsilon_B(i,j)})^a \right],
\] where $\lambda_C = \vert \{(i,j) \in C \mid i \in V\} \vert.$
\end{proof}

\begin{lem} 
\label{new}
Suppose $X = \{\eeh, \ffw, z\}$, where both $w$ and $h$ are greater than zero.
Let $D = \{e_1, \ldots, e_d\} \times \{e_{d+1}, \ldots, e_h, z\}$, where $0 \leq d \leq h$, be either a block in $\sehz\times \sehz$ or the empty set. Then a polynomial $p \in \qyx$ is equivalent in $R$ to a sum of terms of the form $\alpha_h \alpha_w$, where $\alpha_h \in \qqq[Y^-_D(\sehz)]$ and $\alpha_w \in \qqq[Y(\sfwz)]$.
\end{lem}
\begin{proof} First use the relations $[y^-_{e_i f_j}]=[y^-_{e_i z}+ y^-_{z f_j}]$ and $[y^+_{e_i f_j}]=[y^+_{f_j z}- y^-_{z e_i}]$ to rewrite $[p]$ using only elements of $\qyehz$ and $\qyfwz$. We must show that the images in $R$ of elements of $\qqq[Y^+_D(\sehz)]$ have representatives that are sums of elements in $\qqq[Y^-_D(\sehz)]$ and in $\qyfwz$. 

Suppose $i,j \in \seh$, and observe that $[\ypij] = [y^-_{jz}+y^-_{zk}+y^+_{zk} - y^-_{zi}]$, where we may choose $k \in \sfw$. Observe further that $[y^+_{iz}]=[y^-_{iz}+y^-_{zk}+y^+_{zk}]$, where again we may choose $k \in \sfw$. This proves the lemma in the case where $D$ is empty. If $D$ is nonempty, let $B \in \bx$ be a block with $D \subseteq B$, an dconsider the images of the above equations under the isomorphism $f_B$ defined in Remark \ref{isomflip}:
\begin{align*}
\left[-\epsilon_V(i)y^{-\epsilon_B(i,j)}_{ij}\right] &= \left[-\epsilon_V(j)y^{\epsilon_B(j,z)}_{jz}-y^{\epsilon_B(z,k)}_{zk}-y^{-\epsilon_B(z,k)}_{zk}+y^{\epsilon_B(z,i)}_{zi}\right] \\
\left[-\epsilon_V(i)y^{-\epsilon_B(i,z)}_{iz}\right] &= \left[-\epsilon_V(i)y^{\epsilon_B(i,z)}_{iz}-y^{\epsilon_B(z,k)}_{zk}-y^{-\epsilon_B(z,k)}_{zk}\right]
\end{align*}
Since $D \subseteq B$, we know that $\epsilon_B(i,j) = \epsilon_D(i,j)$ whenever $i,j \in \sehz$. Thus these relations allow us to rewrite $[y^{-\epsilon_D(i,j)}_{ij}]$ using only elements of $\qqq[Y^-_D(\sehz)]$ and $\qyfwz$.  \end{proof}

\begin{lem} \label{zroworoct} Let $X$ be a finite subset of $\nnn$ with $\vert X \vert \geq 2$, and let $B=V\times V^c \in \bx$. Let $z \in V^c$. Let $p \in \qyx$ be a homogeneous polynomial of degree at least $2gm(m-1)-m+1-(m-1)(m-2)g$. Then $[p]$ has a representative in $\qqq[Y^-_B(X) \cup \{ y^{-\epsilon_V(i)}_{iz} \mid i \in \xmz\}]$; furthermore, this representative can be chosen to be a linear combination of monomials, each of which, when factorised as $qr$ with $q \in \qymbx$ and $r \in \qqq[\{y^{-\epsilon_V(i)}_{iz}\mid i \in X \setminus \{z\}]$, either satisfies $\deg q \geq m(m-1)g-m+2$, or $r=\theta \prod_{i \in \xmz}(y^{-\epsilon_V(i)}_{iz})^{2g}$ for some monomial $\theta$.  \end{lem}

\begin{proof} We first claim that for any $ i \in \xmz$, the element $[y^{-\epsilon_V(i)}_{iz}]$ together with the images in $R$ of the elements of $Y^-_B(X)$ form a set of generators for $R$. To see this, fix $i \in \xmz$. Since $z \in V^c$, we know that $y^{\epsilon_V(i)}_{iz} \in \ymbx$, and so both $y^+_{iz}$ and $y^-_{iz}$ are in $\ymbx \cup \{ y^{-\epsilon_V(i)}_{iz}\}$. Observe that $[\ypij + \ymij]=[y^+_{iz}+y^-_{iz}]$, and for any $j \in X \setminus \{z, i\}$, either $\ypmij \in \ymbx$ or $\ymij \in \ymbx$. Thus both $[\ypij]$ and $[\ymij]$ are elements of $\langle [\ymbx], [y^{-\epsilon_V(i)}_{iz}] \rangle \subset R$. Similarly, given any $l\neq j \in X$, either $y^+_{lj} \in \ymbx$ or $y^-_{lj} \in \ymbx$. But $[y^+_{lj}+y^-_{lj}]=[\ypij + \ymij]$, so $[y^+_{lj}]$ and $[y^-_{lj}]$ are both in $\langle [\ymbx ], [y^{-\epsilon_V(i)}_{iz}]\rangle $. Finally, observe that $[y^+_{ll}]=[y^+_{lj}+y^-_{lj}]$, so each $[y^+_{ll}]$ is also in $\langle [\ymbx],[y^{-\epsilon_V(i)}_{iz}]\rangle$.  

Without loss of generality, suppose $X=[m]$ and $z=m$. Pick a representative for $[p]$ in $\qqq[\ymbx \cup \{y^{-\epsilon_V(1)}_{1z}\}]$, and factorise each term as $[q_1(y^{-\epsilon_V(1)}_{1z})^{d_1}]$, with $q_1 \in \qymbx$. If $\deg q_1$ is at least $m(m-1)g-m+2$, this term has the desired form; else
\begin{align*}
d_1 &\geq 2gm(m-1)-m+1-(m-1)(m-2)g-m(m-1)g+m-1 \\ & = 2g(m-1) \\ &\geq 2g.
\end{align*} In this case, pick a representative for $[q_1(y^{-\epsilon_V(1)}_{1z})^{d_1}]$ that is a sum of terms of the form $(y^{-\epsilon_V(1)}_{1z})^{2g}q_2(y^{-\epsilon_V(2)}_{2z})^{d_2}$, where $q_2 \in \qymbx$, and consider each term separately. Any term where $\deg q_2 \geq m(m-1)g-m+2$ now has the desired form; else $d_2 \geq 2g(m-1)-2g=2g(m-2)\geq 2g$. Repeat this process: for each term, either $\deg q_i \geq m(m-1)g-m+2$, or $d_i \geq 2g(m-i)$. Thus any term with $\deg q_{m-1} < m(m-1)g-m+2$ has $d_i \geq 2g$ $\forall$ $1 \leq i \leq m-1$, and so is a multiple of $\prod_{i \in X \setminus \{z\}}(y^{-\epsilon_V(i)}_{iz})^{2g}$.
\end{proof}

\begin{lem} 
\label{monster} 
Let $X$ be a finite subset of $\nnn$ with $\vert X \vert \geq 2$, and suppose $p \in \qyx$ is a monomial of degree at least $2g\vert X \vert (\vert X \vert -1)-\vert X \vert +1$. Then we can find a homogeneous polynomial $\chi_\varnothing \in \qyx$, together with polynomials $\psi_B, \chi_B \in \qyx$ for each block $B \in \bx$, such that 
\[
[p]=\left[ \sum_{B \in \bx} \psi_B \prod_{(i,j) \in B}(\ypij \ymij)^{2g} + \sum_{B \in \bx \cup \{\varnothing\}}\chi_B \prod_{(i,j) \in B}(\ymij)^{2g}\prod_{\substack{(i,j) \notin B \cup \bar{B} \\ i < j}}(\ypij)^{2g}\right].
\]
\end{lem}

\begin{ill}
Our goal is to use the relations between the $c_1(\lpmij)$ to express any monomial $[p] \in R$ as a sum of images of ``good'' collections of sections defined at the end of Section \ref{purple}. While we are working in $R$, we will illustrate monomials appearing in our argument with pictures of the form $\rho(x)$ takes in the vanishing loci of the corresponding sections $\spmijx$. Though $\rho(x) \in SO(2n+1)$, our illustrations will be $n\times n$ grids, in which we represent $2 \times 2$ blocks of the form $\left(\begin{array}{cc}x&-y\\y&x\end{array}\right)$ by \begin{tikzpicture} \node[draw, rectangle, red] at (0,0){$+$}; \end{tikzpicture}, the $2 \times 2$ blocks of the form $\left(\begin{array}{cc}z&w\\w&-z\end{array}\right)$ by \begin{tikzpicture} \node[draw, rectangle, blue] at (0,0){$-$}; \end{tikzpicture}, and the $2 \times 2$ blocks of zeros (in the vanishing loci of $(s^+_{ij}(x), s^-_{ij}(x)$) by \begin{tikzpicture} \node[draw, circle, violet] at (0,0){$0$}; \end{tikzpicture}. We have chosen to illustrate our argument using blocks $B$ with the property that $i<j$ $ \forall (i,j) \in B$, since these produce the simplest visualisations; note that this is not an assumption we make in the proof. Lemma \ref{monster} says we can express any monomial of degree $\geq 2gn(n-1)-n+1$ as a sum of terms whose corresponding sections have vanishing loci taking one of the forms shown in Figure \ref{monsterstatement}.
\begin{figure}
\includegraphics[width=\linewidth, trim=50 520 50 50, clip]{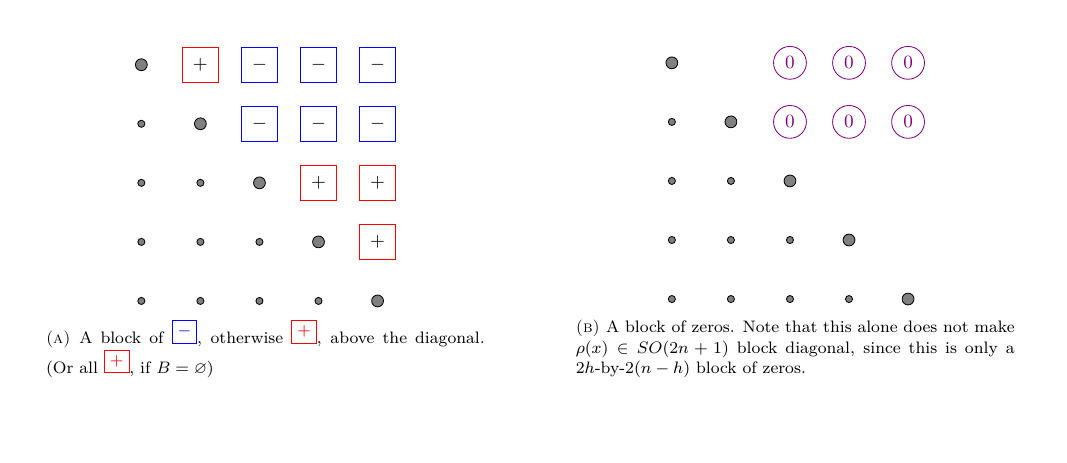}
\caption{}
\label{monsterstatement}
\end{figure}
\end{ill}

\begin{proof} By induction on $\vert X \vert$.

For $\vert X \vert = 2$, without loss of generality take $X = \{1,2\}$. Then 
\[
\qyx=\qqq[y^+_{12},y^-_{12},y^+_{21} ,y^-_{21}].
\]
Since $[y^+_{12}]=[y^+_{21}]$ and $[y^-_{12}]=[-y^-_{21}]$ in $R$, any monomial $p \in \qyx$ of degree at least $4g-1$ is equivalent in $R$ to $\lambda(y^+_{12})^a(y^-_{12})^b$, where $\lambda \in \qqq$ and $a+b \geq 4g-1$. So either $a \geq 2g$ or $b \geq 2g$. If $a \geq 2g$, take $\chi_\varnothing = (y^+_{12})^{a-2g}(y^-_{12})^b$, and $\psi_B=\chi_B=0$ for all $B \in \bx$. If $b \geq 2g$, take $\chi_{(1,2)}=(y^+_{12})^{a}(y^-_{12})^{b-2g}$, and $\psi_B=\chi_\varnothing =0$ for all $B \in \bx$. (\checkmark)

Now suppose $\vert X \vert = m \geq 3$. By Lemma \ref{oddsomez}, there is some $z \in X$ for which when we factorise $p$ as $q_z r_z$, with monomials $q_z \in \qyxmz$ and $r_z \in \qyzx$, the degree of $q_z$ is at least $2g(m-1)(m-2)-m+2$. By the inductive hypothesis, there exist homogeneous polynomials $\tilde{\psi_C}, \tilde{\chi_C}, \tilde{\chi_\varnothing} \in \qyxmz$, for all $C \in \bxmz$, such that 
\[
[q_z]= \left[ \sum_{C \in \bxmz}\tilde{\psi_C}\prod_{(i,j) \in C}(\ypij \ymij )^{2g} + \sum_{C \in \bxmz \cup \{ \varnothing \}} \tilde{\chi_C} \prod_{(i,j) \in C}(\ymij)^{2g}\prod_{\substack{(i,j) \in (X\setminus \{z\}) \times (X\setminus \{z\}) \setminus (C \cup \bar{C}) \\ i < j }}(\ypij)^{2g}\right].
\] 
We will consider these terms separately, as it suffices to show each term has the desired form.

Fix a block $C \in \bxmz$, say $C = \seh \times \sfw$ (so $h+w=m-1$). Consider the term
\begin{equation} \label{codim1zeros}
\left[ r_z \tilde{\psi_C}\prod_{(i,j) \in C}(\ypij \ymij)^{2g}\right]
\end{equation} in $[p]$ (see Figure \ref{codim1boz}). 
\begin{figure}
\includegraphics[width=\linewidth, trim=50 520 50 50, clip]{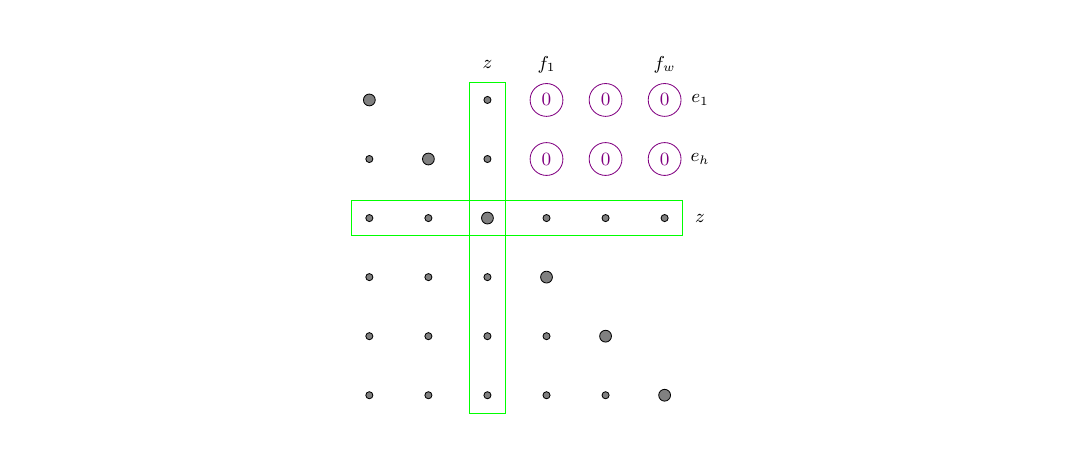}
\caption{The inductive hypothesis gave us a block of zeros in $\bxmz$, using a monomial of degree $4gwh$}
\label{codim1boz}
\end{figure}
Observe that $\deg r_z \tilde{\psi_C} \geq 2gm(m-1)-m+1-4gwh$. Since $[\ypij]=[y^+_{iz} + y^{-}_{zj}]$ and $[\ymij] = [y^-_{iz} + y^-_{zj}]$, each term of $r_z \tilde{\psi_C}$ can be expressed as a sum of terms of the form $[p_h p_w]$, where $p_h \in \qqq[Y(\{ \eeh, z\})]$ and $p_w \in \qqq[Y(\{\ffw, z\})]$ are monomials. By Lemma \ref{indind}, either $\deg p_h \geq 2gh(h+1)-h$, or $\deg p_w \geq 2gw(w+1)-w$ (see figure \ref{codim1bozwtriangles}). 
\begin{figure}
\includegraphics[width=\linewidth, trim=50 520 50 50, clip]{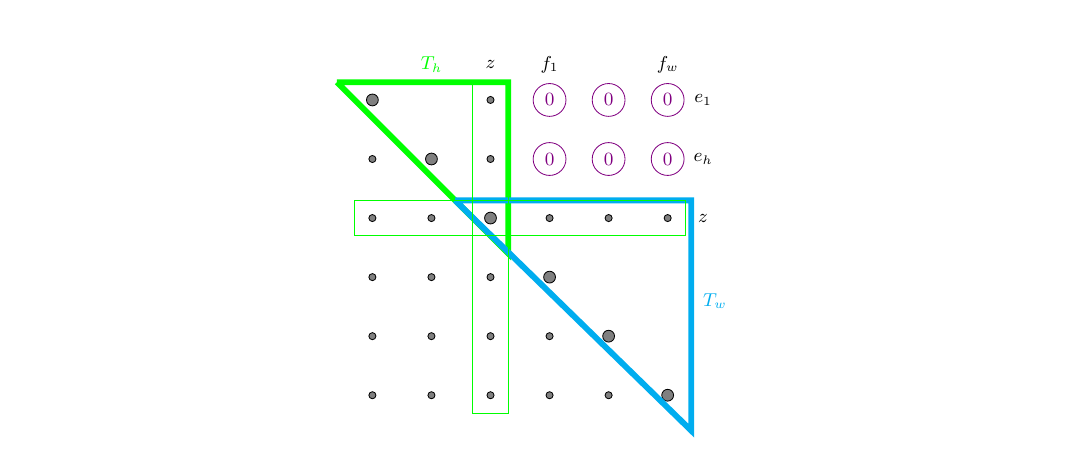}
\caption{We can apply the inductive hypothesis to one of these triangles $T_h$, $T_w$}
\label{codim1bozwtriangles}
\end{figure}
Consider one such monomial in the sum, and without loss of generality assume the former. By the inductive hypothesis, we can find homogeneous polynomials $\psi_D, \chi_D$ for each $D \in \behz$, together with $\chi_\varnothing$, with $\psi_D, \chi_D, \chi_\varnothing \in \qyehz$, such that
\[
[p_h] = \left[ \sum_{D \in \behz} \psi_D \prod_{(i,j) \in D}(\ypij \ymij)^{2g} + \sum_{D \in \behz \cup \{\varnothing \}}\chi_D \prod_{(i,j) \in D} (\ymij)^{2g} \prod_{\substack{i,j \in \sehz \\ (i,j) \notin D \cup \bar{D} \\ i < j}} (\ypij)^{2g} \right]
\] (see figure \ref{indhypto1tri}).
\begin{figure}
\includegraphics[width=\linewidth, trim=50 500 50 50, clip]{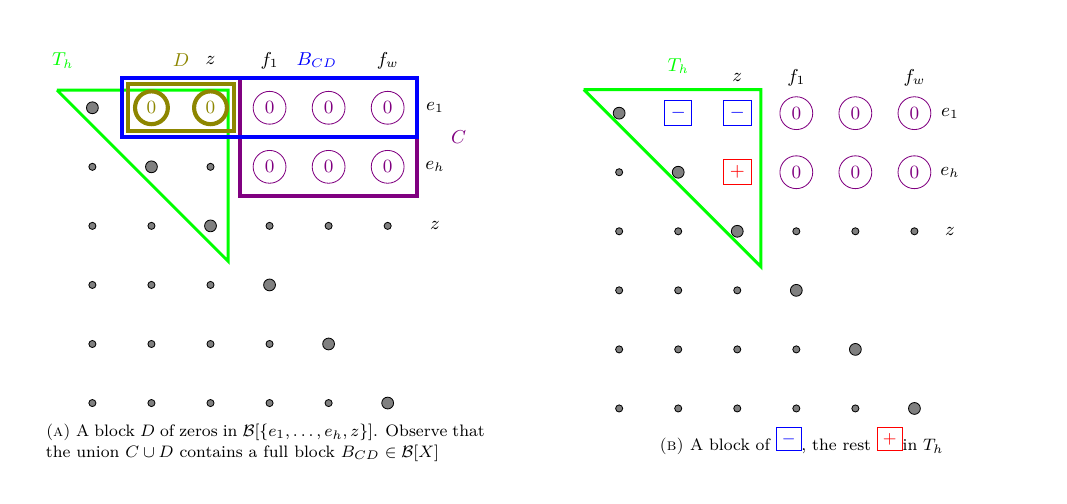}
\caption[Applying the inductive hypothesis to $T_h$]{Applying the inductive hypothesis to $T_h$, we find either a block of zeros, or a block of \squareminus with the rest \squareplus}
\label{indhypto1tri}
\end{figure}

Fix a block $D = \{e_1, \ldots e_d\}\times \{e_{d+1}, \ldots, e_h, z\}$ (some $1 \leq d \leq h$), and consider the term 
\[
\left[p_w \psi_D \prod_{(i,j) \in D }(\ypij \ymij)^{2g} \prod_{(i,j) \in C}(\ypij \ymij)^{2g}\right]
\] of $[p]$. By Lemma \ref{cd1sbextends}, there is a block $B_{CD} \in \bx$ with $B_{CD} \subset C\cup D \cup \bar{D}$, so our term is a multiple of $\prod_{(i,j) \in B_{CD}}(\ypij \ymij)^{2g}$ and hence has the desired form (see Figure \ref{indhypto1tri}A).

Now fix $D = \{e_1, \ldots, e_d\} \times \{e_{d+1}, \ldots, e_h,z\}$ where $0 \leq d \leq h$ (that is, allow $D$ to be empty), and consider the term
\[ 
\left[p_w \chi_D \prod_{(i,j) \in D}(\ymij)^{2g}\prod_{\substack{i,j \in \sehz \\ (i,j) \notin D \cup \bar{D} \\ i < j}}(\ypij)^{2g}\prod_{(i,j) \in C}(\ypij \ymij) ^{2g}\right]
\] of $[p]$ (see Figure \ref{indhypto1tri}B). 
By Lemma \ref{new}, we may express $[p_w \chi_D]$ as a sum of terms of the form $[\alpha_w \alpha_h]$, where $\alpha_w \in \qyfwz$ and $\alpha_h \in \qqq[Y^-_D(\sehz)]$. By Lemma \ref{indoroct}, either $\deg \alpha_w \geq 2gw(w+1)-w$ or $\deg \alpha_h \geq h(h+1)g -(h+1) +2$ (see figure \ref{octorindhyp}).
\begin{figure}
\includegraphics[width=\columnwidth, trim=50 480 50 50, clip]{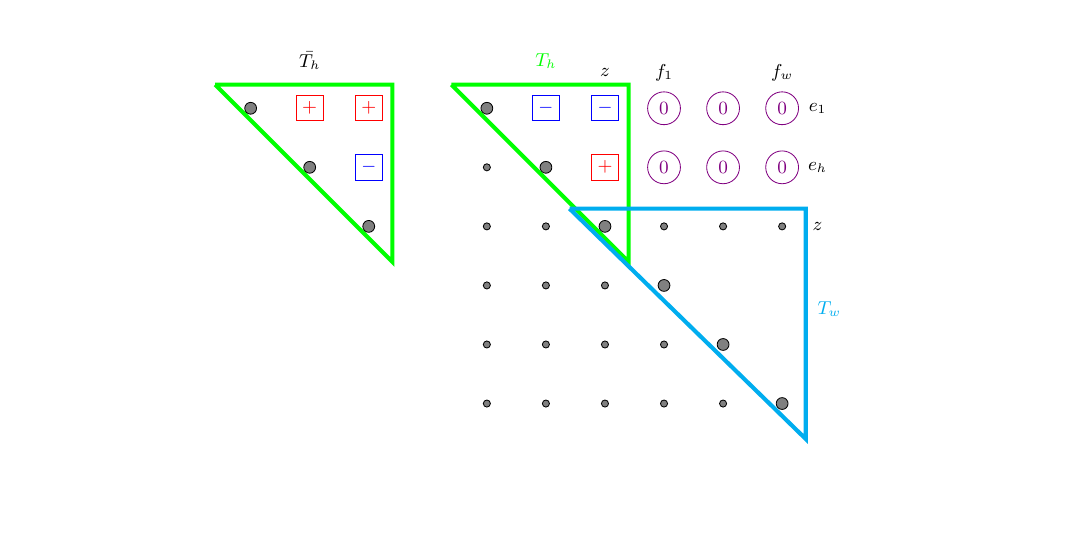}
\caption{We can apply either Corollary \ref{ajpropflipblock} to $\bar{T_h}$ (left), or the inductive hypothesis to $T_w$}
\label{octorindhyp}
\end{figure}
If $\deg \alpha_h \geq h(h+1)g -(h+1)+2$, then by Corollary \ref{ajpropflipblock}, for each block $E \in \behz$ we can find a monomial $\beta_E$ such that 
\[
[\alpha_h] = \left[ \sum_{E \in \bbb[\sehz]}\beta_E \prod_{(i,j) \in E}(y^{\epsilon_D(i,j)}_{ij})^{2g} \right].
\] 
\begin{figure}
\includegraphics[width=\linewidth, trim=50 520 50 50, clip]{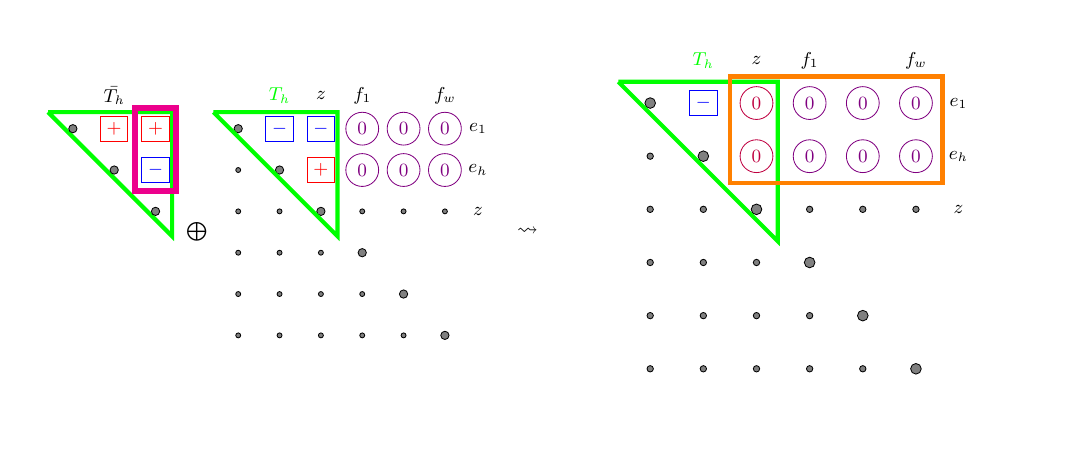}
\caption[Applying Corollary \ref{ajpropflipblock}]{Applying Corollary \ref{ajpropflipblock} to $\overline{T_h}$ gives a  block $E$ of \squareplus or \squareminus as shown in $\overline{T_h}$, which combines with the \squareplus and \squareminus in $T_h$ to get zeros everywhere in $E$. The union $C \cup E$ contains a full block $B_{CE} \in \bx$}
\label{blockint1makesblock}
\end{figure}
Then for fixed $E$, the term 
\[\left[\alpha_w \beta_E \prod_{(i,j)\in E}(y^{\epsilon_D(i,j)}_{ij})^{2g}\prod_{(i,j) \in D}(\ymij)^{2g}\prod_{\substack{i, j \in \sehz \\ (i,j) \notin D \cup \bar{D} \\ i<j}}(\ypij)^{2g}\prod_{(i,j) \in C}(\ypij \ymij ) ^{2g}\right]
\] of $[p]$ contains the factor 
\[\prod_{(i,j) \in E} ( \ypij \ymij)^{2g} \prod_{(i,j) \in C} (\ypij \ymij)^{2g}
\](up to replacing $y^+_{ji}$ with $\ypij$ whenever $(i,j) \in E$, $(i,j) \notin D \cup \bar{D}$, and $i<j$). By Lemma \ref{cd1sbextends}, $C \cup E \cup \bar{E}$ contains a block $B_{CE} \in \bx$, so this term has the desired form (see Figure \ref{blockint1makesblock}). 

Suppose instead that $\deg \alpha_w\geq 2gw(w+1)-w$. Then by the inductive hypothesis, 
\[[\alpha_w] = \left[ \sum_{F \in \bfwz} \theta_F \prod_{(i,j) \in F} (\ypij \ymij)^{2g} + \sum_{F \in \bfwz \cup \{\varnothing\}} \phi_F \prod_{(i,j) \in F}(\ymij)^{2g}\prod_{\substack{i,j \in \sfwz \\ (i,j) \notin F \cup \bar{F} \\ i<j}}(\ypij)^{2g} \right],
\] for some homogeneous polynomials $\theta_F$, $\phi_F \in \qyfwz$. 
\begin{figure}
\includegraphics[width=\linewidth, trim=50 450 50 50, clip]{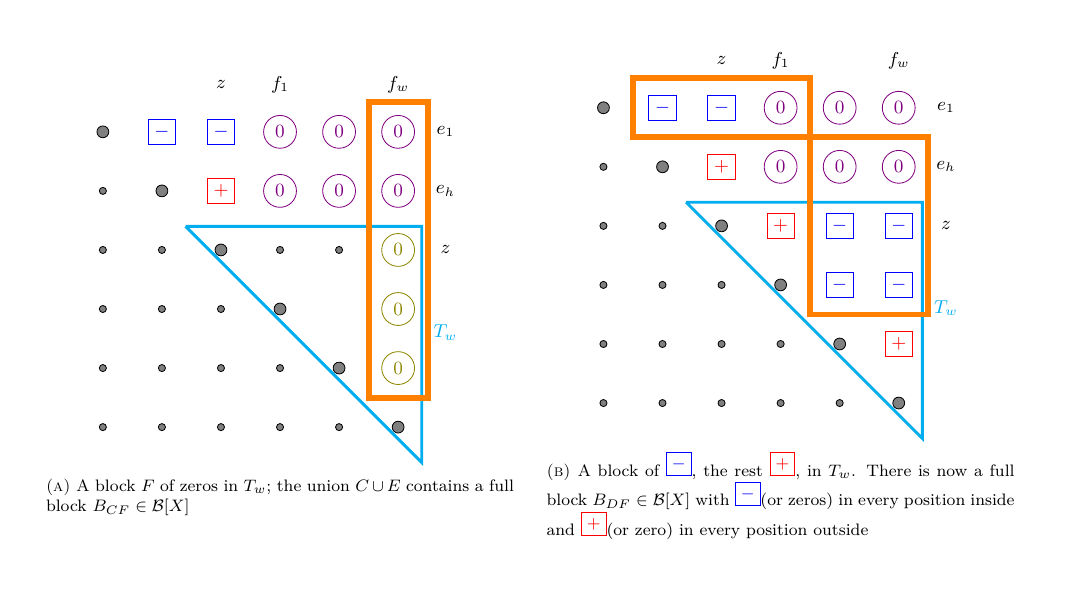}
\caption{Applying the inductive hypothesis in $T_w$ results in one of these two situations}
\label{caseonedone}
\end{figure}

Fix a block $F \in \bfwz$ and consider each term separately. By Lemma \ref{cd1sbextends}, $C \cup F \cup \bar{F}$ contains a block $B_{CF} \in \bx$, and so the term 
\[
\left[\alpha_h \theta_F \prod_{(i,j) \in F} (\ypij \ymij)^{2g} \prod_{(i,j) \in D}(\ymij)^{2g}\prod_{\substack{i,j \in \sehz \\ (i,j) \notin D \cup \bar{D} \\ i<j}} \prod_{(i,j) \in C}(\ypij \ymij)^{2g} \right]
\] of $[p]$ contains the factor $\prod_{(i,j) \in B_{CF}}(\ypij \ymij)^{2g}$ and hence is of the desired form (see Figure \ref{caseonedone}A).

By Lemma \ref{ebfbextendinb}, there exists a block $B_{DF} \in \bx$ with either $D \cup F \subset B_{DF} \subset D \cup F \cup C \cup \bar{C}$, or $\bar{D} \cup F \subset B_{DF} \subset \bar{D} \cup F \cup C \cup \bar{C}$, and thus the term 
\[
\left[ \alpha_h \phi_F \prod_{(i,j) \in F}(\ymij)^{2g}\prod_{\substack{i,j \in \sfwz \\ (i,j) \in F \cup \bar{F} \\ i < j}}(\ypij)^{2g} \prod_{(i,j) \in D}(\ymij)^{2g} \prod_{\substack{i,j \in \sehz \\ (i,j) \notin D \cup \bar{D} \\ i<j}}(\ypij)^{2g} \prod_{(i,j) \in C} (\ypij \ymij)^{2g}\right]
\] of $[p]$ contains the factor $\prod_{(i,j) \in B_{DF}}(\ymij)^{2g}\prod_{\substack{i,j \in X \\ (i,j) \notin B_{DF} \cup \bar{B_{DF}} \\ i<j}}(\ypij)^{2g}$, and hence is of the desired form (see Figure \ref{caseonedone}B).
Hence each term $ \left[ r_z \tilde{\psi_C} \prod_{(i,j ) \in C}(\ypij \ymij)^{2g} \right]$ has a representative in $\qyx$ of the desired form. 

It remains to show that  
\[ \left[ r_z \sum_{C \in \bxmz \cup \{\varnothing\}} \tilde{\chi_C} \prod_{(i,j) \in C}(\ymij)^{2g} \prod_{\substack{i,j \in \xmz \\ (i,j) \notin C \cup \bar{C} \\ i<j}}(\ypij)^{2g} \right]
\] (Figure \ref{codim1bomrp}) has a representative in $\qyx$ of the desired form; again it suffices to show this for each term separately.
\begin{figure}
\includegraphics[width=\linewidth, trim=50 520 50 50, clip]{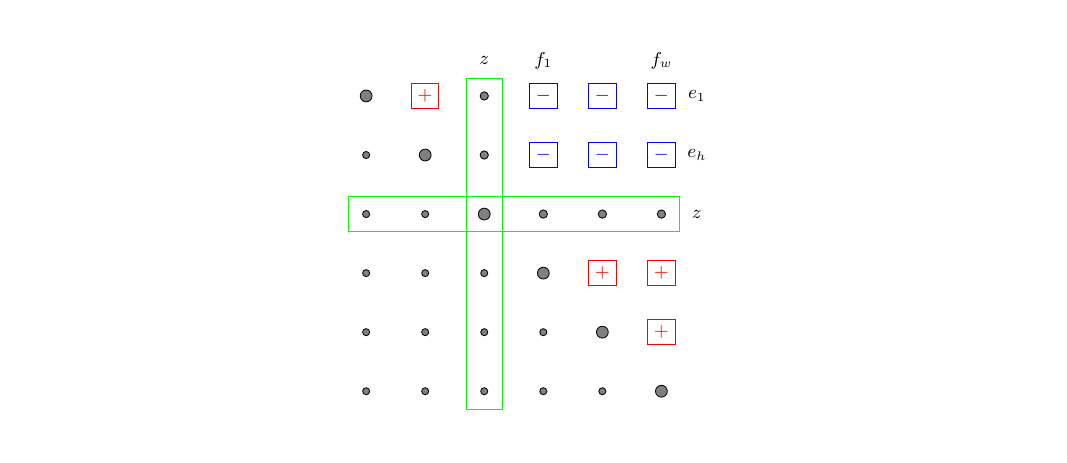}
\caption[The result of the inductive hypothesis]{The inductive hypothesis gave us a block $C \in \bxmz$of \squareminus, and \squareplus everywhere else above the diagonal and away from $z$}
\label{codim1bomrp}
\end{figure}

Fix $C = V \times ((\xmz) \setminus V) \in \bxmz$ and consider 
\begin{equation}\label{rzchitildeetc}
r_z \tilde{\chi_C} \prod_{(i,j) \in C}(\ymij)^{2g}\prod_{\substack{i,j \in \xmz \\ (i,j) \notin C \cup \bar{C} \\ i<j}}(\ypij)^{2g}.\end{equation}
Observe that $\deg r_z \tilde{\chi_C} \geq 2gm(m-1)-m+1 -g(m-1)(m-2)$. By Lemma \ref{zroworoct}, $[r_z \tilde{\chi_C}]$ has a representative in $\qyx$ that is a sum of terms of the form $q \prod_{i \in \xmz}(y^{-\epsilon_V(i)}_{iz})^{d_i}$, where either $\deg q \geq m(m-1)g-m+2$, or $d_i \geq 2g$ $ \forall i \in \xmz$ (Figure \ref{case2done}).
  
\begin{figure}
\includegraphics[width=\linewidth, trim=50 480 50 50, clip]{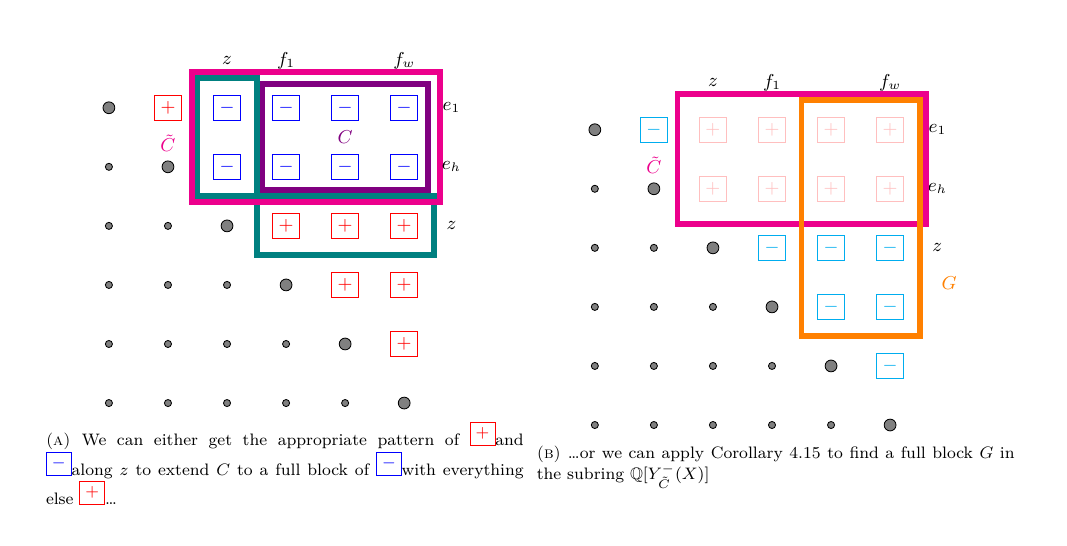}
\caption{Lemma \ref{zroworoct} puts us in one of these situations}
\label{case2done}
\end{figure}

In the latter case, the corresponding terms of $[p]$ have the factor 
\[
\prod_{(i,j) \in C}(\ymij)^{2g}\prod_{\substack{i,j \in \xmz \\ (i,j) \notin C \cup \bar{C} \\ i<j}}(\ypij)^{2g}\prod_{i \in \xmz}(y^{-\epsilon_V(i)}_{iz})^{2g} = \prod_{(i,j) \in \tilde{C}}(\ymij)^{2g}\prod_{\substack{i,j \in X \\ (i,j) \notin \tilde{C} \cup \bar{\tilde{C}}\\ i<j}}(\ypij)^{2g},
\]where $\tilde{C} = V \times (X\setminus V)$ is the extension of the block $C$ obtained by adding $z$ to the second factor. Hence these terms have the desired form (see Figure \ref{case2done}A). 

If a term $q \prod_{i \in \xmz}(y^{-\epsilon_V(i)}_{iz})^{d_i}$ has $\deg q \geq m(m-1)g-m+2$, then by Corollary \ref{ajpropflipblock} there exist homogeneous polynomials $\gamma_G$, for each $G \in \bx$, such that 
\[
[q] = \left[ \sum_{G \in \bx} \gamma_G \prod_{(i,j) \in G} (y^{\epsilon_{\tilde{C}}(i,j)}_{ij})^{2g}\right]
\]
(see Figure \ref{case2done}B). Then each such term \[\prod_{(i,j) \in C}(\ymij)^{2g}\prod_{\substack{i,j \in \xmz \\ (i,j) \notin C \cup \bar{C} \\ i<j}}(\ypij)^{2g}\gamma_G\prod_{(i,j) \in G}(y^{\epsilon_{\tilde{C}}(i,j)}_{ij})^{2g}\] in \eqref{rzchitildeetc} contains the factor $\prod_{(i,j) \in G \vert_{\xmz}}(\ypij \ymij)^{2g}$ (up to replacing $y^+_{ji}$ by $\ypij$ whenever $(i,j) \notin \tilde{C} \cup \bar{\tilde{C}}$, $(i,j) \in G$ and $i>j$). By Remark \ref{restrictblock}, $G \vert_{\xmz} \in \bxmz$, and thus we are reduced to a monomial of the form \eqref{codim1zeros} (Figure \ref{case2now1}), which we dealt with above.

Thus we have found a representative for $[p]$ in $\qyx$ that has the desired form.
\begin{figure}
\includegraphics[width=\linewidth, trim=50 520 50 50, clip]{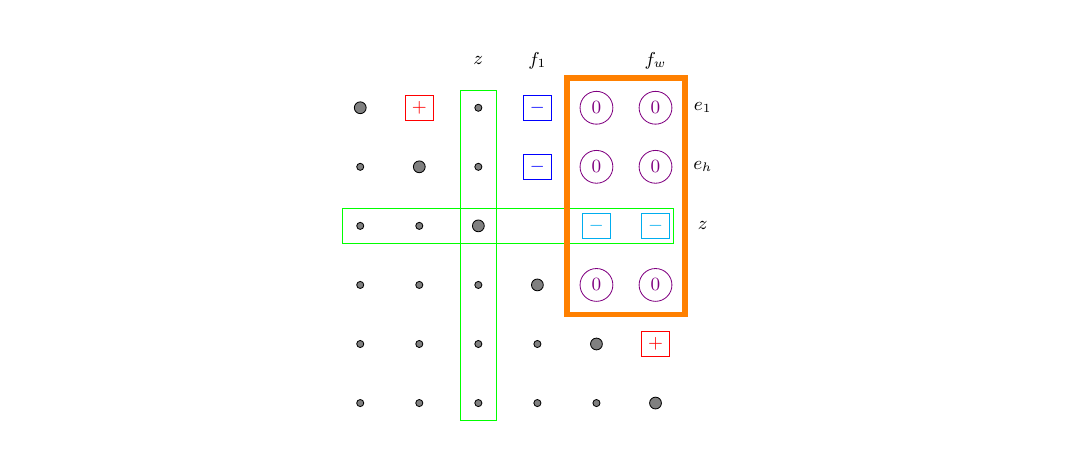}
\caption[Combining the block $G$ with ]{Combining the block $G$ (Figure \ref{case2done}B) with \squareminus everywhere in $C$ and \squareplus elsewhere away from $z$ (Figure \ref{codim1bomrp}) gives a block of zeros in $\bxmz$, reducing us to the earlier case (Figure \ref{codim1boz})}
\label{case2now1}
\end{figure}
\end{proof}

\begin{lem} \label{octandonediag} Let $B \in \mathcal{B}[[n]]\cup \{ \varnothing \}$ be a block (or the empty set), let $p \in \qyx$ be a monomial, and fix $1 \leq i \leq n$. Then $[p]$ is equivalent in $R$ to a sum of terms of the form $[qr]$, where $q \in \qqq[y^+_{ii}]$ and $r \in \qymbx$. 
\end{lem}
\begin{proof}We must check that the images in $R$ of the elements of $\ymbx$ together with $[y^+_{ii}]$ generate $R$. We showed in the proof of Lemma \ref{zroworoct} that the images of elements in $\ymbx$ together with $[y^{-\epsilon_B(i,z)}_{iz}]$ (for some $z \neq i$ generate $R$, so it suffices to show that $[y^{-\epsilon_B(i,z)}_{iz}] \in \langle [\ymbx], [y^+_{ii}]\rangle$. But this is clear, because $y^{\epsilon_B(i,z)}_{iz} \in \ymbx$, and $[y^{\epsilon_B(i,z)}_{iz}+y^{-\epsilon_B(i,z)}_{iz}]=[y^+_{ii}]$.
\end{proof}

\begin{notn} If $p \in \qyx$ is a monomial, let $d^\pm_{ij}$ (for $i,j \in X$) be the integers such that $p=\lambda\prod_{i,j \in X}(\ypmij)^{d^\pm_{ij}}$ (where $\lambda \in \qqq$), and define $d_{ij}(p):=d^+_{ij}+d^-_{ij}+d^+_{ji}+d^-_{ji}$. 
\end{notn}

\begin{lem} \label{procedure} Let $X=[n]$ and let $p \in \qyx$ be a monomial of degree at least $2gn^2+\frac{1}{2}(n-1)(n-2)-2gn-n(n-1)g$. Let $C \in \bx \cup \{\varnothing\}$ be a fixed block (or the empty set). Then we can find a (finite) set $\mathcal{D}$ whose elements are $2g$-tuples of (possibly empty) unions of blocks in $\bx$, and homogeneous polynomials $\phi_D$ for each $D \in \mathcal{D}$ and $\theta_B$ for each $B \in \bx$, such that 
\[[p]= \left[\sum_{B \in \bx} \theta_B \prod_{(i,j) \in B} (y^{\epsilon_C(i,j)}_{ij})^{2g} + \sum_{D \in \mathcal{D}}\left(\phi_D\prod_{m=1}^{2g}\left(\prod_{(i,j) \in D_m}y^{\epsilon_C(i,j)}_{ij}\prod_{\substack{(i,j) \notin D_m \cup \bar{D_m} \\ i<j}}y^{-\epsilon_C(i,j)}_{ij}\right)\right) \right].
\]
\end{lem}

\begin{proof} First factorise $p$ as $p_-p_+$, where $p_- \in \qqq[Y^-_C(X)]$ and $p_+ \in \qqq[Y^+_C(X)]$. If $\deg p_- \geq n(n-1)g-n+2$, we are done by Corollary \ref{ajpropflipblock}. Otherwise, let $a <2g$ be the largest integer such that $\deg p_- \geq \frac{1}{2}n(n-1)a-n+2$, and write 
\[[p_-] = \left[\sum_{B \in \bx} \psi_B \left(\prod_{m=1}^a \prod_{(i,j) \in B} y^{\epsilon_C(i,j)}_{ij}\right) \right]
\]
(which is possible by Corollary \ref{ajpropflipblock}). Each monomial in the resulting polynomial has the form $\alpha \beta \gamma$, where $\alpha \in \qqq[Y^-_C(X)]$, $\beta = \prod_{m=1}^s (\prod_{(i,j) \in D_m}y^{\epsilon_C(i,j)}_{ij})\in \qqq[Y^-_C(X)]$ for some $0 \leq s \leq 2g$ and unions $D_m$ of blocks in $\bx$, and $\gamma \in \qqq[Y^+_C(X)]$. Consider monomials of this form. 

Case 1: If $\deg(\alpha \beta) \geq n(n-1)g-n+2$, then $[\alpha \beta \gamma]$ has a representative of the desired form by Corollary \ref{ajpropflipblock}.

Case 2: If $s=2g$ and there is a block $B \in \bx$ such that $B \subseteq D_m$ for every $1 \leq m \leq 2g$, then $\beta$ contains the factor $\prod_{(i,j) \in B}(y^{\epsilon_C(i,j)}_{ij})^{2g}$, so $\alpha \beta \gamma$ is of the desired form.

Case 3: Observe further that for each $i < j$, we know $\yecij \in \qqq[Y^-_C(X)]$ and $\ymecij \in \qqq[Y^+_C(X)]$. Thus, if $d_{ij}(\beta \gamma)\geq 2g$ $\forall i \neq j$, then swapping $\ypij$ with $y^+_{ji}$ or $\ymij$ with $-y^-_{ji}$ as necessary and taking $D_m = \varnothing$ for $s+1 \leq m \leq 2g$, we may write 
\[[\beta \gamma] = \left[\lambda  \prod_{m=1}^{2g}\left( \prod_{(i,j) \in D_m} \yecij \prod_{\substack{(i,j) \notin D_m \cup \bar{D_m}\\i<j}}\ymecij \right) \right],
\]
thus finding a representative for $[\alpha \beta \gamma]$ of the desired form. 

The following procedure takes monomials of degree at least $2gn^2+ \frac{1}{2}(n-1)(n-2)-2gn-n(n-1)g$, and finds representatives for them as sums of monomials, each of which falls into one of the three cases above.

Start with a monomial $\alpha \beta \gamma$ of degree $\geq 2gn^2+\frac{1}{2}(n-1)(n-2)-2gn-n(n-1)g$, where $\alpha \in \qqq[Y^-_C(X)]$, $\beta = \prod_{m=1}^s(\prod_{(i,j) \in D_m}\yecij)$ for some $0 \leq s \leq 2g$ and unions $D_m$ of blocks in $\bx$, and $\gamma \in \qqq[Y^+_C(X)]$. 
\begin{enumerate} 
	\item \label{applycor} If $\deg(\alpha\beta )\geq n(n-1)g-n+2$, we are in Case 1. Stop. Otherwise, go to Step 2. 
	\item \label{checkalpha} If $\deg(\alpha) \geq \frac{1}{2}n(n-1)-n+2$, go to Step \ref{hitalpha}. Otherwise, go to Step 4. 
	\item \label{hitalpha} Apply Corollary \ref{ajpropflipblock} to write $[\alpha]=\left[\sum_{B \in \bx}\phi_B\prod_{(i,j) \in B} \yecij\right]$. Consider each term in the resulting polynomial separately.
	\begin{enumerate}
		\item If $s<2g$: Replace $\alpha$ with $\phi_B$. Set $D_{s+1}:=B$. Replace $\beta$ with $\beta\cdot \prod_{(i,j) \in B}\yecij$. Go to Step \ref{checkalpha}. 
		\item If $s=2g$: 
		\begin{enumerate} 
			\item If $B \subseteq D_m$ for all $1 \leq m \leq 2g$, then we are in Case 2. Stop. 
			\item Otherwise, pick $1 \leq m \leq 2g$ with $B \not\subseteq D_m$. Replace $D_m$ with $B \cup D_m$ in $\beta$. Go to Step \ref{checkalpha}. 
		\end{enumerate} 
	\end{enumerate}
	\item 
	\begin{enumerate}
		\item If $d_{ij}(\beta \gamma)\geq 2g$ for all $i<j$, we are in Case 3. Stop. 
		\item \label{moveinaxles} Otherwise, pick $i<j$ with $d_{ij}(\beta \gamma) < 2g$. We know $\deg(\alpha) \leq \frac{1}{2}n(n-1)-n+1$, whilst $\deg (\alpha \beta \gamma) \geq 2gn^2+\frac{1}{2}(n-1)(n-2)-2gn-n(n-1)g$, so $\deg(\beta \gamma) \geq n(n-1)g=(2g)\binom{n}{2}$. Thus there must be some $k,l \in X$ with $d_{kl}(\beta \gamma) > \begin{cases} 2g & k \neq l,\\ 0 & k=l \end{cases}$ and $y^{-\epsilon_C(k,l)}_{kl}$ must be a factor of $\gamma$. Replace one instance of $y^{-\epsilon_C(k,l)}_{kl}$ in $\gamma$ with $y^{\epsilon_C(k,i)}_{ki}+y^{-\epsilon_C(i,j)}_{ij}+y^{\epsilon_C(l,j)}_{l,j}$ (the two are equivalent in $R$), and treat each term in the resulting polynomial separately. Go to Step \ref{applycor}.
	\end{enumerate}
\end{enumerate}
Observe that Step \ref{hitalpha} increases the number $b$ of blocks (counted with multiplicity) in $\cup_{r=1}^sD_r$, and Step \ref{moveinaxles} either increases $\deg(\alpha\beta)$ or decreases $d:=\sum_{\substack{i \neq j \\ d_{ij} < 2g}}(2g-d_{ij})$, for each term in the polynomials these steps create. None of the steps decrease $b$ or $\deg(\alpha\beta)$, or increase $d$.  If $\deg(\alpha \beta) \geq n(n-1)g-n+2$, or $d=0$, or $b$ gets large enough to force every $D_m$ to contain some common block $B$ ($b \geq (2g-1)\vert \bx \vert +1$ will do), then the algorithm terminates. So even when a step results in a polynomial with more than one term, there are still only finitely many terms, each of which are also closer to a terminating condition than the previous monomial. So this procedure terminates in finite time, giving us a representative of the desired form for any monomial of sufficient degree.
\end{proof}

\begin{prop}\label{alltogethernow} Let $ X \subset \nnn$ with $\vert X \vert = n$. Let $p \in \qyx$ be a homogeneous polynomial of degree at least $2gn^2 + \frac{1}{2}(n-1)(n-2)$. Then for each $B = V \times V^c \in \bx$ and each $C \in \bx \cup \{\varnothing\}$, we can find \begin{itemize} \item homogeneous polynomials $\theta_B, \phi_B, \psi_C \in \qyx$ \item elements $p_V \in \mathcal{A}_{\vert V \vert}$ and $p_{V^c} \in \mathcal{A}_{\vert V^c \vert}$ \item bijections $f_V:[\vert V \vert ] \to V$ and $f_{V^c}:[\vert V^c\vert] \to V^c$ \item a finite set $\mathcal{D_C}$ whose elements are $2g$-tuples of unions of blocks in $\bx$, and \item homogeneous polynomials $\chi_D \in \qyx$ for each $D \in \mathcal{D}_C$ \end{itemize} such that 
\begin{multline}\label{desiredform}[p]=[\sum_{B = V \times V^c \in \bx}\prod_{(i,j) \in B}(\ypij\ymij)^{2g}(\theta_B\cdot \alpha ((f_V)_*(p_V))+\phi_B \cdot \alpha((f_{V^c})_*(p_{V^c})))\\+\sum_{C \in \bx \cup \{\varnothing\}}\psi_C \sum_{D \in \mathcal{D}_C}\chi_D  \prod_{m=1}^{2g}\prod_{(i,j) \in D_m}\ypij\ymij \prod_{(i,j) \notin D_m \cup \bar{D_m}}y^{-\epsilon_C(i,j)}_{ij}].
\end{multline}
\end{prop}

\begin{rmk} By the definition of the sets $\mathcal{A}_n$ (see Definition \ref{recursiveAsets}), if this relation holds, then there exist $p_i \in \mathcal{A}_n$, monomials $\theta_i \in \qyx$, and a bijection $f_X:[n] \to X$ such that $[p] = [ \sum_i \theta_i \cdot \alpha (f_X)_*(p_i)]$.
\end{rmk}

\begin{proof} By induction on $n$. 

For $n=1$, if $p \in \qqq[y^+_{11}]$ is a monomial of degree $\geq 2g$, it is in the desired form. (\checkmark)

Suppose $n \geq 2$. By Lemma \ref{monster}, we can write
\[
[p]= \left[ \sum_{B \in \bx} \psi_B \prod_{(i,j) \in B} (\ypij \ymij)^{2g} + \sum_{B \in \bx \cup \{\varnothing\}} \chi_B \prod_{(i,j) \in B}(\ymij)^{2g}\prod_{\substack{(i,j) \notin B \cup \bar{B} \\ i<j}}(\ypij)^{2g}\right].
\]

As usual, we treat each term separately. Fix a block $B =V \times V^c \in \bx$ and consider a monomial $\psi_B \prod_{(i,j) \in B } ( \ypij \ymij)^{2g}$. Suppose $B = \seh \times \sfw$; then $\deg \psi_B \geq 2gn^2 + \frac{1}{2}(n-1)(n-2) -4gwh$. By using the relations $[y^\pm_{e_if_j}]=\frac{1}{2}[y^+_{e_ie_i}\pm y^+_{f_jf_j}]$, we can express $[\psi_B]$ as a sum of terms of the form $[\psi_h \psi_w]$, where $\psi_h \in \qyeh$ and $\psi_w \in \qyfw$ are monomials of total degree $\geq 2gn^2+\frac{1}{2}(n-1)(n-2)-4gwh$. By Lemma \ref{bozsqind}, in each of these terms either $\deg \psi_h \geq 2gh^2+\frac{1}{2}(h-1)(h-2)$ or $\deg \psi_w \geq 2gw^2 + \frac{1}{2}(w-1)(w-2)$, and we can apply the inductive hypothesis to the appropriate monomial. Applying the inductive hypothesis to $\psi_h$, we find that there exist $q_i \in \mathcal{A}_n$, monomials $\zeta_i \in \qqq[Y(V)]$, and a bijection $f_V: [h] \to V$ such that $[\psi_h] = [\sum_i \zeta_i \cdot \alpha(f_V)_*(q_i)]$; applied to $\psi_w$ the inductive hypothesis gives $[\psi_w] = [\sum_j \xi_j \cdot \alpha(f_{V^c})_*(r_j)]$, where $r_j \in \mathcal{A}_n$, $\xi_j \in \qqq[Y(V^c)]$, and $f_{V^c}:[w] \to V^c$ is a bijection. Applying the inductive hypothesis to the appropriate factor in each term of $\psi_B$ and summing the resulting representatives, we obtain 
\[
\left[\psi_B \prod_{(i,j) \in B}(\ypij \ymij)^{2g}\right]=\left[\theta_B \cdot \alpha((f_V)_*(p_V))+ \phi_B \cdot \alpha((f_{V^c})_*(p_{V^c}))\right]
\] 
for homogeneous polynomials $\theta_B, \phi_B \in \qyx$ and elements $p_V \in \mathcal{A}_h$ and $p_{V^c} \in \mathcal{A}_w$. This is in the desired form \eqref{desiredform} (taking $\psi_C = 0$ $\forall C \in \bx \cup \{ \varnothing \}$).

Now fix $C \in \bx \cup \{\varnothing\}$ and consider the term $\chi_C \prod_{(i,j) \in C}(\ymij)^{2g}\prod_{\substack{(i,j) \notin C \cup \bar{C} \\ i<j}}(\ypij)^{2g}$. Observe that $\deg \chi_C \geq 2gn^2 + \frac{1}{2} (n-1)(n-2) - n(n-1)g$. By repeated application of Lemmas \ref{2goroct} and \ref{octandonediag}, we can write $[\chi_C] = \left[ \prod_{i\in X}(y^+_{ii})^{2g} \theta + \phi r \right]$, where $r \in \qqq[Y^-_C(X)]$ is a homogeneous polynomial with $\deg r \geq n(n-1)g-n+2$, and $\theta, \phi \in \qyx$. By Corollary \ref{ajpropflipblock}, $[r] = \left[\sum_{E \in \bx} \theta_E \prod_{(i,j) \in E} (y^{\epsilon_C(i,j)}_{ij})^{2g} \right]$, for homogeneous polynomials $\theta_E \in \qyx$.  Thus 
\[\left[ \phi r \prod_{(i,j) \in C}(\ymij)^{2g}\prod_{\substack{(i,j) \notin C \cup \bar{C} \\ i<j }}(\ypij)^{2g}\right] = \left[ \sum_{E \in \bx} \psi_E \prod_{(i,j) \in E} (\ypij \ymij)^{2g} \right]
\] (for some $\phi_E \in \qyx$), which is of the form considered above. 

Finally, consider the monomial $\prod_{i \in X}(y^+_{ii})^{2g} \theta \prod_{(i,j) \in C}(\ymij)^{2g}\prod_{(i,j) \notin C \cup \bar{C}}(\ypij)^{2g}$. Note that $\deg \theta \geq 2gn^2 + \frac{1}{2}(n-1)(n-2) - \frac{1}{2}n(n-1)g-2gn$. Applying Lemma \ref{procedure}, 
\[
[\theta] = \left[ \sum_{F \in \bx} \phi_F \prod_{(i,j) \in F} (y^{\epsilon_C(i,j)}_{ij})^{2g} + \sum_{D \in \mathcal{D}_C}\psi_D \prod_{m=1}^{2g}(\prod_{(i,j) \in D_m}y^{\epsilon_C(i,j)}_{ij})(\prod_{\substack{(i,j) \notin D_m \cup \bar{D_m} \\ i<j}}y^{-\epsilon_C(i,j)}_{ij}) \right].
\]
As usual, consider each monomial individually. Fix $F$, and consider 
\[
\prod_{i \in X}(y^+_{ii})^{2g}\prod_{(i,j) \in C}(\ymij)^{2g}\prod_{\substack{(i,j) \notin C \cup \bar{C}\\i<j}}(\ypij)^{2g}\phi_C \prod_{(i,j) \in F}(y^{\epsilon_C(i,j)}_{ij})^{2g}.
\] This contains the factor $\prod_{(i,j) \in F}(\ypij \ymij)^{2g}$, and so is of the form discussed above. The remaining term 
\[
\prod_{i \in X} (y^+_{ii})^{2g} \prod_{(i,j) \in C}(\ymij)^{2g}\prod_{\substack{(i,j) \notin C \cup \bar{C} \\ i<j}}(\ypij)^{2g} \left( \sum_{D \in \mathcal{D}_C}\psi_D \prod_{m=1}^{2g}\prod_{(i,j) \in D_m}y^{\epsilon_C(i,j)}_{ij}\prod_{\substack{(i,j) \notin D_m \cup \bar{D_m} \\ i<j}}y^{-\epsilon_C(i,j)}_{ij}\right)
\] is already in the form \eqref{desiredform}. 
\end{proof}
\subsection{Proof of the main theorem}
\begin{thm} For each $\phi \in \Phi(G)$, let $k_\phi$ be a nonnegative integer. Then the cohomology class 
$
\prod_{\phi \in \Phi(G)}(c_1(L_\phi))^{k_\phi} \in H^{2 \sum_{\phi \in \Phi(G)}k_\phi}(S_{n,g}(t); \qqq)
$
 vanishes whenever $\sum_{\phi \in \Phi(G)} k_\phi \geq 2gn^2 + \frac{1}{2}(n-1)(n-2)$.   
\end{thm}
\begin{proof} Let $X = [n]$ and consider the rings $\qyx$ and $R = \qyx / I$. Let $J \subset H^*(S_{n,g}(t); \qqq)$ be the subring generated by the $c_1(L_\phi)$ for $\phi \in \Phi(G)$. Since the relations \eqref{relns} hold in $J$, the map 
\begin{align*} \pi: R & \to J \\ [y^\pm_{ij}] & \mapsto c_1(\lpmij)
\end{align*} 
defines a ring homomorphism. Consider the element $\prod_{\phi \in \Phi(G)} c_1(L_\phi)^{k_\phi} \in J$. It has a representative $[\prod(y^\pm_{ij})^{d^\pm_{ij}}]$ in $R$. Suppose $\sum_\phi k_\phi \geq 2gn^2 + \frac{1}{2}(n-1)(n-2)$. Then by Proposition \ref{alltogethernow}, $[p]$ is equivalent in $R$ to an expression of the form \eqref{desiredform}. So $\pi(p)$ vanishes in $H^*(S_{n,g}(t);\qqq)$ by Corollary \ref{goodthingsvanish}.
\end{proof}

\bibliographystyle{plain}

\end{document}